\documentclass[10pt,twoside]{homg3} 

\usepackage{amssymb}
\usepackage{amsmath}
\usepackage{amsfonts}
\usepackage{graphicx}
\usepackage{tikz}
\usetikzlibrary{arrows}

\newcommand{\image}{{\mathrm{Im\,}}}

\newtheorem{proposition}[theorem]{Proposition}
\newtheorem{example}[theorem]{Example}
\newtheorem{notation}[theorem]{Notation}

\def\TituloCorto{The contact structure in the space of light rays }
\def\AutorCorto{A. Bautista, A. Ibort, J. Lafuente}

\begin{document}

\tittle{The Contact Structure in the Space of Light Rays\hspace{-1mm}\thanks{This work has been partially supported by the Spanish MINECO project MTM2014-54692 and QUITEMAD+, S2013/ICE-2801.}}

\def\authors{\aaa{Alfredo {\surname BAUTISTA}$\, ^{(a)}$,
\ Alberto {\surname IBORT}$\, ^{(a)}$
\ and  Javier {\surname LAFUENTE}$\, ^{(b)}$
}}

\def\direc{\address{
$(a)$ Departmento de Matem\'aticas\\
Universidad Carlos III de Madrid\\
28040 Madrid, Spain \\
abautist@math.uc3m.es, albertoi@math.uc3m.es \\ \vskip .2cm
$(b)$ Departmento de Geometr\'{\i}a y Topolog\'{\i}a \\
Universidad  Complutense de
Madrid \\
28040 Madrid, Spain\\
lafuente@mat.ucm.es}}
\maketitle

\begin{center}
{\em Dedicado a Jose Mar\'{\i}a Montesinos con motivo de su jubilaci\'on.}
\end{center}


\begin{abstract}
\noindent The natural topological, differentiable and geometrical structures on the space of light rays of a given spacetime are discussed.   The relation between the causality properties of the original spacetime and the natural structures on the space of light rays are stressed.   Finally, a symplectic geometrical approach to the construction of the canonical contact structure on the space of light rays is offered.
\end{abstract}


\par
\medskip


\section{Introduction}
In the recent articles \cite{Ba14,Ba15} it was shown that causality relations on spacetimes can be described alternatively in terms of the geometry and topology of the space of light rays and skies.   This alternative description of causality, whose origin can be traced back to Penrose, was pushed forward by R. Low \cite{Lo88,Lo06} and, as indicate above, largely accomplished in the referred works by Bautista, Ibort and Lafuente.

Shifting the point of view from ``events'' to ``light rays'' and ``skies''  to analyze causality relations has deep and worth discussing implications.    Thus, for instance, as Low himself noticed \cite{Lo90, Lo94}, in some instances, two events are causally related iff the corresponding skies are topologically linked and, a more precise statement of this fact, constitutes the so called Legendrian Low's conjecture (see for instance \cite{Na04, Ch10}).

The existence of a canonical contact structure on the space of light rays plays a cornerstone role in this picture. Actually it was shown in \cite{Ba15} that two events on a strongly causal spacetime are causally related iff there exists a non-negative sky Legendrian isotopy relating their corresponding skies.

Moreover, and as an extension of Penrose's twistor programme, it would be natural to describe attributes of the conformal class of the Lorentzian metric such as the Weyl tensor, in terms of geometrical structures on the space of light rays.

It is also worth to point out that in dimension 3 this dual approach to causality becomes very special.  Actually if the dimension of spacetime is $m= 3$, the dimension of the space of light rays is also $2m-3 = 3$ and the skies are just Legendrian circles.  Even more because the dimension of the contact distribution is 2 the space of 1-dimensional subspaces is the 1-dimensional projective space $\mathbb{RP}^1$, hence the curve defined by the tangent spaces to the skies along the points of a light ray defines a projective segment of it and it is possible to define Low's causal boundary \cite{Lo06} unambiguously.    We will not dwell into these matters in the present paper and we will leave it to a detailed discussion elsewhere. (see also A. Bautista Ph. D. Thesis \cite{Ba15b}).

Because of all these reasons we have found relevant to describe in a consistent and uniform way the fundamental structures present on the space of light rays of a given spacetime: that is, its topological, differentiable and geometrical structures, the later one, exemplified by its canonical contact structure.
Thus, the present work will address in a sistematic and elementary way the description and construction of the aforementioned structures, highlighting the interplay between the causility properties of the original spacetime and the structures of the corresponding space of light rays.

The paper will be organized as follows.  Section 2 will be devoted to introduce properly the space of light rays and in Sect. 3 its natural differentiable structure will be constructed.   The tangent bundle of the space of light rays and the canonical contact structure on the space of light rays will be the subject of Sects. 4 and 5.

Looking for brevity in the present work, because a self--contained article will become a very long text, we suggest to the readers references \cite{On83}, \cite{BE96}, \cite{HE73}, \cite{Mi08}  and \cite{Pe79} in order to review the basic elements of causality theory in Lorentzian manifolds.


\section{The space $\mathcal{N}$ of light rays and its differentiable structure} \label{sec:seccion-estruc-dif}

\subsection{Constructions of the space $\mathcal{N}$} 
Given a spacetime $\left(M,\mathbf{g}\right)$, i.e., a Hausdoff, time-oriented, Lorentzian smooth manifold, we will denote by $\mathbb{N}$ the subset of all null vectors of $\widehat{T}M$, where in what follows, $\widehat{T}M$ will be used to make reference to the bundle resulting of eliminating the zero section of $TM$, that is $\widehat{T}M=\left\{ v\in TM:v\neq 0 \right\}$.
The zero section of $TM$ separates both connected components of $\mathbb{N}$ that will be denoted by
\begin{equation*}
\mathbb{N}^+=\left\{ v\in \mathbb{N}: v \text{ future }\right\} \, , \qquad \mathbb{N}^-=\left\{ v\in \mathbb{N}: v \text{ past }\right\} \, .
\end{equation*}%
We will call the fibres $\mathbb{N}_p$, $\mathbb{N}^+_p$ and $\mathbb{N}^-_p$ \emph{lightcone}\index{lightcone}, \emph{future lightcone}\index{lightcone!future} and \emph{past lightcone}\index{lightcone!past} at $p\in M$ respectively.     We define the \emph{set of light rays} of $\left(M,\mathbf{g}\right)$ by
\[
\mathcal{N}_{\mathbf{g}}=\left\{ \image \left(\gamma\right)\subset M:\gamma \text{ is a maximal null geodesic in }\left(M,\mathbf{g}\right) \right\}
\]
where $\image \left(\gamma\right)$ denotes the image of the curve $\gamma$. This definition seems to depend on the metric $\mathbf{g}$, however we will show that the space of light rays depends only on the conformal class of the spacetime metric.

We define the \emph{conformal class of metrics in $M$ equivalent to $\mathbf{g}$}\index{conformal class!of metrics}  by
\[
\mathcal{C}_{\mathbf{g}}=\left\{ \overline{\mathbf{g}}\in \mathfrak{T}_0^2\left(M\right): \overline{\mathbf{g}}=e^{2\sigma}\mathbf{g},\hspace{2mm}0<\sigma\in\mathfrak{F}\left(M\right) \right\}
\]
and we call $\left(M,\mathcal{C}_{\mathbf{g}}\right)$ the corresponding \emph{conformal class of spacetimes equivalent to $\left(M,\mathbf{g}\right)$}\index{conformal class!of spacetimes}.

It is known that two metric are conformally equivalent if the lightcones coincide at every point (see \cite[Prop. 2.6 and Lem. 2.7]{Mi08} or \cite[p. 60-61]{HE73}). This fact can be automatically translated to the spaces of light rays defined by two different metrics on a manifold $M$. The following proposition brings to light the equivalence among spaces of light rays. 

\begin{proposition}\label{prop00100}
Let $\left( M,\mathbf{g}\right)$ and $\left( M,\overline{\mathbf{g}}\right)$ be two spacetimes and let  $\mathcal{N}_{\mathbf{g}}$ and $\mathcal{N}_{\overline{\mathbf{g}}}$ be their corresponding spaces of light rays.
Then $\left( M,\mathbf{g}\right)$ and $\left( M,\overline{\mathbf{g}}\right)$ are conformally equivalent if and only if $\mathcal{N}_{\mathbf{g}} = \mathcal{N}_{\overline{\mathbf{g}}}$.
\end{proposition}

Because of Prop, \ref{prop00100} we have the following definition.

\begin{definition}
Let $\left( M,\mathcal{C}_{\mathbf{g}}\right)$ be a conformal class of spacetimes with $\dim M=m\geq 3$.
We will call \emph{light ray}\index{light ray} the image $\gamma\left(I\right)$ in $M$ of a maximal null geodesic $\gamma:I\rightarrow M$ related to any metric $\bar{\mathbf{g}}\in\mathcal{C}_{\mathbf{g}}\left(M\right)$.
It will be denoted by $\left[\gamma\right]$ or $\gamma$ when there is not possibility of confusion, that is $\left[\gamma\right] \in \mathcal{N}$, $\gamma \in \mathcal{N}$ or also $\gamma \subset M$.
So, every light ray is equivalent to an unparametrized null geodesic.
Then, we will say that the \emph{space of light rays $\mathcal{N}$ of a conformal class of spacetimes}\index{space!of light rays} $\left(M,\mathcal{C}_{\mathbf{g}}\right)$ is the set
\begin{equation*}
\mathcal{N}=\left\{ \gamma\left(I\right)\subset M \mid \gamma:I \rightarrow M \text{ is a maximal null geodesic for any metric }  \bar{\mathbf{g}}\in \mathcal{C}_{\mathbf{g}} \right\} \, .
\end{equation*}
\end{definition}

A more geometric construction of $\mathcal{N}$ is possible as a quotient space of the tangent bundle $TM$ \cite{Lo06}.
This construction will allow us to show how $\mathcal{N}$ inherits the topological and differentiable structures of $TM$.

Let us consider the \emph{geodesic spray}\index{geodesic!spray} $X_{\mathbf{g}}$ related to the metric $\mathbf{g}$, that is the vector field in $TM$ such that its integral curves define the geodesics in $\left(M,\mathbf{g}\right)$ and their tangent vectors, and the \emph{Euler field $\Delta$ in $TM$}\index{field!Euler!in $TM$}\index{Euler!field in $TM$}, the vector field in $TM$ whose flow are scale transformations along the fibres of $TM$.

It is easy to see that both $X_{\mathbf{g}}$ and $\Delta$ are tangent to $\mathbb{N}$, and the differentiable distribution in $\mathbb{N}^+$ given by $\mathcal{D}=\mathrm{span}\left\{ X_{\mathbf{g}},\Delta\right\} $ verifies
\begin{equation*}
\left[ \Delta,X_{\mathbf{g}}\right] =X_{\mathbf{g}} \, ,
\end{equation*}%
hence $\mathcal{D}$ is involutive and, by Fr\"{o}benius' Theorem \cite[Thm. 1.60]{Wa83}, it is also integrable.
This means that the quotient space $\mathbb{N}^+/\mathcal{D}$ is well defined.
Every leaf of $\mathcal{D}$ is the equivalence class consisting of a future--directed null geodesic and all its affine reparametrizations preserving time--orientation, hence the space
\begin{equation*}
\mathcal{N}^+=\mathbb{N}^+/\mathcal{D} \, .
\end{equation*}
is the space of future-oriented light rays of $M$.
In a similar way we may construct the space of past-oriented light rays $\mathcal{N}^- = \mathbb{N}^-/\mathcal{D}$.

The space of light rays $\mathcal{N}$ can be obtained as the quotient $\mathbb{N} /\widetilde{\mathcal{D}}$
where $\widetilde{D}$ denotes the scale transformation group acting on $\mathbb{N}$, that is, $v \mapsto \lambda v$, $\lambda \neq 0$, $v \in \mathbb{N}$.  Notice that $\widetilde{\mathcal{D}} \cong \mathbb{R} - \{ 0 \}$ and the orbits of the connected component containing the identity can be identified with the leaves of $\mathcal{D}$.

Because our spacetime $M$ is time-orientable, in what follows we will consider the space of future-oriented light rays $\mathcal{N}^+$ and we will denote it again as $\mathcal{N}$ without risk of confusion.   As it will be shown later on, this convention will be handy as the space $\mathcal{N}^+$ carries a co-orientable contact structure (see Sect. \ref{sec:contact}) whereas $\mathcal{N}$ does not.

\subsection{Differentiable structure of $\mathcal{N}$}

If we require the space of light rays of $M$ to be a differentiable manifold, it is necessary to ensure that the leaves of the distribution that builds $\mathcal{N}$, are regular submanifolds.
This is not automatically true for any spacetime $M$, as example \cite[Ex. 1]{Lo01} shows, so it will be necessary to impose further conditions to ensure it.

It important to observe that if $M$ is strongly causal, every causal curve has a finite number of connected components in a relatively compact neighbourhood, and therefore the distribution $\mathcal{D}=\mathrm{span}\left\{ X_{\mathbf{g}}, \Delta \right\}$ is regular, then the quotient $\mathcal{N}=\mathbb{N}^{+} / \mathcal{D}$ is a manifold, as the following proposition \cite[Pr. 2.1]{Lo89} claims.

\begin{proposition}\label{prop00150}
Let $M$ be a strongly causal spacetime, then the distribution above defined by $\mathcal{D}=\mathrm{span}\left\{ X_{\mathbf{g}}, \Delta \right\}$ is regular and the space of light rays $\mathcal{N}$ inherits from $\mathbb{N}^{+}$ the structure of differentiable manifold such that $p_{\mathbb{N}^{+}} :\mathbb{N}^{+}\rightarrow \mathcal{N}$ defined by $p_{\mathbb{N}^{+}}\left(u\right)=\left[\gamma_u\right]$ is a submersion.
\end{proposition}

The space $\mathcal{N}$ can also be constructed as a quotient of the bundle of null directions $\mathbb{PN}$ defined below, that is we can proceed to compute the quotient space $\mathbb{N}^+/\mathcal{D}$ in two steps, we will first quotient with respect to the action dilation field $\Delta$ and, secondly we will pass to the quotient defined by the integral curves of the geodesic field.

First, consider the involutive distribution $\mathcal{D}_{\Delta}=\mathrm{span}\left\{\Delta\right\}$. Observe that for any $u\in TM$ the vectors $e^t u \in TM $ run the integral curve of $\Delta$ passing through $u$ when $t\in \mathbb{R}$. Then, if
\begin{equation}
\phi \colon TV\rightarrow \mathbb{R}^{2m};\quad \xi \mapsto \left( x,\dot{x}\right)
\label{equation3.1}
\end{equation}
are the canonical coordinates in $TM$, then $\left( x,e^t\dot{x}\right)$ describe integral curves of $\Delta$. Taking homogeneous coordinates for $\dot{x}$, we can obtain that $\left( x,\left[\dot{x}\right]\right)$ are coordinates adapted to foliation generated by $\mathcal{D}_{\Delta}$. Therefore $\mathcal{D}_{\Delta}$ is regular and its restriction to $\mathbb{N}^{+}$ is also regular, hence the quotient space $\mathbb{N}^{+}/\mathcal{D}_{\Delta}$ defined by
\[
\mathbb{PN}=\mathbb{N}^{+}/\mathcal{D}_{\Delta}=\left\{ \left[\xi\right]:\eta\in \left[\xi\right] \Leftrightarrow \eta=e^t \xi \text{ for some } t\in \mathbb{R} \text{ and }\xi\in\mathbb{N}^{+}  \right\}
\]
is a differentiable manifold and, moreover, the canonical projection $\pi^{\mathbb{N}^{+}}_{\mathbb{PN}}:\mathbb{N}^{+} \rightarrow \mathbb{PN}$, given by $\xi \mapsto \left[\xi\right]$, is a submersion.

The next step is to find a regular distribution that allows us to define $\mathcal{N}$ by a quotient.
For each vector $u\in \mathbb{N}^{+}_{p}$ there exists a null geodesic $\gamma _{u}$ such that $\gamma _{u}\left( 0\right) =p$ and $\gamma _{u}^{\prime }\left( 0\right) =u$, and given two vectors $u,v\in \mathbb{N}^{+}_{p}$ verifying that $v=\lambda u$ with $\lambda >0$, then the geodesics $\gamma _{u}$ and $\gamma _{v}$ such that $\gamma _{u}\left( 0\right)=\gamma _{v}\left( 0\right) =p$ have the property
\begin{equation*}
\gamma _{v}\left( s\right) =\gamma _{\lambda u}\left( s\right) =\gamma_{u}\left( \lambda s\right)
\end{equation*}
hence they have the same image in $M$ and then $\gamma _{v}=\gamma _{u}$ as unparametrized sets in $M$.
This fact implies that the elevations to $\mathbb{PN}$ of the null geodesics of $M$ define a foliation $\mathcal{D}_{G}$.
Two directions $\left[ u\right] ,\left[ v\right] \in \mathbb{PN}$ belong to the same leaf of the foliation $\mathcal{D}_{G}$ if for the vectors $v\in \mathbb{N}^{+}_{p}$ and $u\in \mathbb{N}^{+}_{q}$ there exist null geodesics $\gamma _{1}$ and $\gamma _{2}$ and values $t_{1},t_{2}\in \mathbb{R}$ verifying
\begin{equation*}
\left\{
\begin{array}{l}
\gamma _{1}\left( t_{1}\right) =p\in M \\
\gamma _{1}^{\prime }\left( t_{1}\right) =v\in \mathbb{N}^{+}_{p}%
\end{array}
\right. \text{\hspace{1cm}and\hspace{1cm}}\left\{
\begin{array}{l}
\gamma _{2}\left( t_{2}\right) =q\in M \\
\gamma _{2}^{\prime }\left( t_{2}\right) =u\in \mathbb{N}^{+}_{q}%
\end{array}%
\right.
\end{equation*}%
such that there is a reparametrization $h$ verifying $\gamma _{1}=\gamma_{2}\circ h$.

Hence, the space of leaves of $\mathcal{D}_{G}$ in $\mathbb{PN}$ coincides with $\mathcal{N}$, that is,
\begin{equation*}
\mathcal{N}=\mathbb{PN} / \mathcal{D}_{G}
\end{equation*}

The map $p_{\mathbb{PN}} : \mathbb{PN} \longrightarrow \mathcal{N}$ given by
$\left[ u\right] \longmapsto \left[\gamma _{u}\right]$ is well defined, since $\gamma _{\lambda u}\left( s\right) =\gamma _{u}\left( \lambda s\right) $ as seen above, and it verifies the identity $p_{\mathbb{PN}} \left( \left[ \gamma _{u}^{\prime }\left( s\right) \right] \right) =\left[\gamma _{u}\right]\in\mathcal{N}$
for all $s$.

Proposition \ref{prop00150} can be formulated for the bundle $\mathbb{PN}$ with regular distribution$\mathcal{D}_{G}$ instead.   In \cite[Thm. 1]{Lo06} a similar result is proved for the subbundle $\left(\mathbb{N}^{*}\right)^{+}$ of the cotangent bundle $T^{*}M$.

Now, we will describe a generic way to construct coordinate charts in $\mathcal{N}$.
First, we define for any subset $W\subset M$, the sets
$\mathbb{N}\left(W\right)=\left\{ \xi\in \mathbb{N}:\pi_{M}^{\mathbb{N}}\left(\xi\right)\in W\subset M\right\}$, and
$\mathbb{PN}\left(W\right)=\left\{ \left[\xi\right]\in \mathbb{PN}:\pi_{M}^{\mathbb{PN}}\left(\left[\xi\right]\right)\in W\subset M\right\}$.

Take $V\subset M$ a causally convex, globally hyperbolic and open set with a smooth spacelike Cauchy surface $C\subset V$.
Let $\mathcal{U}$ be the image of the projection $p_{\mathbb{N}}:\mathbb{N}^{+}\left( V\right) \mapsto \mathcal{N}$.
Since $\mathbb{N}^{+}\left(V\right)$ is open in $\mathbb{N}^{+}$ and $p_{\mathbb{N}}$ is a submersion, then $\mathcal{U}\subset \mathcal{N}$ is open.
Each null geodesic passing through $V$ intersects $C$ in a unique point and since $p_{\mathbb{N}^{+}}=p_{\mathbb{PN}}\circ \pi^{\mathbb{N}^{+}}_{\mathbb{PN}}$, this ensures that
\[
\mathcal{U} = p_{\mathbb{N}} \left( \mathbb{N}^{+}\left( V\right) \right) = p_{\mathbb{N}} \left( \mathbb{N}^{+}\left( C\right) \right) = p_{\mathbb{N}}\circ \pi_{\mathbb{PN}}^{\mathbb{N}^{+}} \left( \mathbb{N}^{+}\left( C\right) \right) = p_{\mathbb{PN}} \left( \mathbb{PN}\left( C\right) \right) = p_{\mathbb{PN}} \left( \mathbb{PN}\left( V\right) \right).
\]
Then we can define the diffeomorphism $\sigma = \left. p_{\mathbb{PN}}\right\vert _{\mathbb{PN}\left( C\right) }:\mathbb{PN}\left( C\right)\mapsto \mathcal{U}$. 
So, we have the following diagram
\begin{equation}\label{diagrama002}
\begin{tikzpicture}[every node/.style={midway}]
\matrix[column sep={6em,between origins},
        row sep={2em}] at (0,0)
{ \node(PN1)   {$\mathbb{PN}\left( V\right)$}  ; & \node(N) {$\mathcal{U}$}; \\
  \node(PN2) {$\mathbb{PN}\left( C\right)$};                   \\};
\draw[->] (PN1) -- (N) node[anchor=south]  {$p_{\mathbb{PN}} $};
\draw[->] (PN2) -- (N) node[anchor=north]  {$\sigma$};
\draw[<-right hook] (PN1)   -- (PN2) node[anchor=east] {inc};
\end{tikzpicture}
\end{equation}
If $\phi$ is any coordinate chart for $\mathbb{PN}\left(C\right)$ then $\phi\circ \sigma^{-1}$ is a coordinate chart for $\mathcal{U}\subset \mathcal{N}$.

Observe that if $M$ is time--orientable, there exists a non--vanishing future timelike vector field $T\in\mathfrak{X}\left(M\right)$ everywhere.
Then we can define the submanifold $\Omega^{T}\left(C\right)\subset \mathbb{N}^{+}\left(C\right)$ by
\[
\Omega^{T}\left(C\right)=\left\{\xi\in \mathbb{N}^{+}\left(C\right): \mathbf{g}\left(\xi, T\right)=-1  \right\}
\]

The restriction $\left.\pi^{\mathbb{N}^{+}}_{\mathbb{PN}}\right|_{\Omega^{T}\left(C\right)}: \Omega^{T}\left(C\right)\rightarrow \mathbb{PN}\left(C\right)$ of the submersion $\pi^{\mathbb{N}^{+}}_{\mathbb{PN}}:\mathbb{N}^{+}\rightarrow \mathbb{PN}$ is clearly a diffeomorphism.

So, we have the following diagram
\begin{equation}\label{cadena-espacios}
\mathcal{N} \supset \mathcal{U} \leftrightarrow \mathbb{PN}\left(C\right) \leftrightarrow \Omega^{T}\left(C\right) \hookrightarrow \mathbb{N}^{+}\left(C\right) \hookrightarrow \mathbb{N}^{+} \hookrightarrow TM
\end{equation}
where $\leftrightarrow$ and $\hookrightarrow$ represent diffeomorphisms and inclusions respectively.

Then, the composition of the diffeomorphism $\mathcal{U} \rightarrow \Omega^{T}\left(C\right)$ with the restriction of a coordinate chart in $TM$ to the vectors in $\Omega^{T}\left(C\right)$, can be used to construct a coordinate chart in $\mathcal{N}$.

If we require the space of light rays of $M$ to be a differentiable manifold, it remains to ensure that $\mathcal{N}$ is a Hausdorff topological space.
Again, it is not verified for any strongly causal spacetime $M$ as we can check in example \ref{N-not-T2}, so we need to state conditions to ensure it.

\begin{example}\label{N-not-T2}
$\mathcal{N}$ is not Hausdorff.
Consider the two--dimensional Minkowski spacetime and remove the point $\left( 1,1\right)$.
Clearly, $M$ is strongly causal.
Let $\left\{ \tau _{n}\right\} \subset \mathbb{R}$ be a sequence such that $\lim\limits_{n\mapsto \infty }\tau _{n}=0$.
Then the sequence of null geodesic given by $\lambda _{n}\left( s\right) =\left(s,\tau _{n}+s\right) $ with $s\in \left( -\infty ,\infty \right) $ converges to two different null geodesics, $\mu _{1}\left( s\right) =\left(s,s\right) $ with $s\in \left( -\infty ,1\right) $ and $\mu _{2}\left( s\right)=\left( s,s\right) $ with $s\in \left( 1,\infty \right)$.
\end{example}

A sufficient condition to ensure that $\mathcal{N}$ is Hasudorff is the absence of naked singularities \cite[Prop. 2.2]{Lo89}, but we will see in example \ref{ejemplo-naked-T2} that it is not a necessary condition.

\begin{example}\label{ejemplo-naked-T2}
Let $\mathbb{M}$ be the 3--dimensional Minkowski space--time described by coordinates $\left(t,x,y\right)$ and equipped with the metric $\mathbf{g}=-dt\otimes dt + dx\otimes dx + dy\otimes dy$. 
Now, consider the restriction $\mathbb{B}=\left\{\left(t,x,y\right)\in\mathbb{M}:t^2 + x^2 + y^2 < 1 \right\}$. It is clear that $\mathbb{B}$ is strongly causal but not globally hyperbolic.


It is easy to find inextensible causal curves fully contained in the chronological future or past of some event $p\in \mathbb{B}$, then $\mathbb{B}$ is nakedly singular.

The space of light rays $\mathcal{N}_{\mathbb{B}}$ of $\mathbb{B}$ verifies that $\mathcal{N}_{\mathbb{B}}\subset \mathcal{N}_{\mathbb{M}}$. Moreover, the inclusion $\mathcal{N}_{\mathbb{B}}\hookrightarrow \mathcal{N}_{\mathbb{M}}$ is clearly injective and every light ray has a neighbourhood in $\mathcal{N}_{\mathbb{B}}$ corresponding to another neighbourhood in $\mathcal{N}_{\mathbb{M}}$. It is not difficult to show that $\mathcal{N}_{\mathbb{B}}$ is open in $\mathcal{N}_{\mathbb{M}}$. Since $M$ is globally hyperbolic, then $\mathcal{N}_{\mathbb{M}}$ is Hausdorff and therefore $\mathcal{N}_{\mathbb{B}}$ is also Hausdorff.
\end{example}

Example \ref{ejemplo-naked-T2} shows that the absence of naked singularities is a condition too strong for a strongly causal spacetime $M$. Moreover in this case, $M$ becomes globally hyperbolic as Penrose proved in \cite{Pe79}.

A suitable condition to avoid the behavior of light rays in the paradigmatic example \ref{N-not-T2} but to permit naked singularities similar to the ones in example \ref{ejemplo-naked-T2} is the condition of \emph{null pseudo--convexity}.

\begin{definition}
A spacetime $M$ is said to be \emph{null pseudo--convex}\index{null!pseudo--convex}\index{pseudo--convex!null} if for any compact $K\subset M$ there exists a compact $K'\subset M$ such that any null geodesic segment $\gamma$ with endpoints in $K$ is totally contained in $K'$.
\end{definition}

In \cite{Lo90-2}, Low states the equivalence of null pseudo--convexity of $M$ and the Hausdorffness of $\mathcal{N}$ for a strongly causal spacetime $M$.
From now on, we will assume that $M$ is a strongly causal and null pseudo--convex spacetime unless others conditions are pointed out.


\section{Tangent bundle of $\mathcal{N}$}

To take advantage of the geometry and topology of $\mathcal{N}$ it is needed to have a suitable characterization of the tangent spaces $T_{\gamma }\mathcal{N}$ for any $\gamma \in \mathcal{N}$.
We will proceed as follows: first, we fix an auxiliary metric in the conformal class  $\mathcal{C}$, then we will define \emph{geodesic variations} (in particular, \emph{variations by light rays}) and \emph{Jacobi fields}, explaining the relation between both concepts (in lemmas \ref{lema00351}, \ref{lema00352}, \ref{lemmaDC92} and proposition \ref{prop00200}). Then, in proposition \ref{prop00254}, we will characterize tangent vector of $TM$ in terms of Jacobi fields. 
Second, we will keep an eye on how the Jacobi fields changes when we vary the parameters of the corresponding variation by light rays or conformal metric of $M$ (see from lemma \ref{lemma-g-geodesic} to \ref{lemma270}).

Finally, in proposition \ref{prop00399}, we will get the aim of this section identifying tangent vectors of $\mathcal{N}$ with some equivalence classes of Jacobi fields.

\begin{definition}
A differentiable map $\mathbf{x}:\left(a,b\right)\times \left(\alpha,\beta\right) \rightarrow M$ is said to be a \emph{variation of a segment of curve $c:\left(\alpha,\beta\right)\rightarrow M$}\index{variation!of a curve} if $c\left(t\right)=\mathbf{x}\left(s_0,t\right)$ for some $s_0\in \left(a,b\right)$.
We will say that $V_{s_0}^{\mathbf{x}}$ is the \emph{initial field of $\mathbf{x}$}\index{initial field!of variation}\index{variation!initial field of} in $s=s_0$ if
\[
V_{s_0}^{\mathbf{x}}\left(t\right) = d \mathbf{x}_{\left(s_0,t\right)}\left(\frac{\partial}{\partial s}\right)_{\left(s_0,t\right)}= \left.\frac{\partial \mathbf{x}\left(s,t\right)}{\partial s}\right|_{\left(s_0,t\right)}
\]
defining a vector field along $c$.

We will say that $\mathbf{x}$ is a \emph{geodesic variation}\index{geodesic!variation}\index{variation!geodesic} if any longitudinal curve of $\mathbf{x}$, that is $c_s^{\mathbf{x}}=\mathbf{x}\left(s,\cdot\right)$ for $s\in \left(a,b\right)$, is a geodesic.

If the longitudinal curves $c_s^{\mathbf{x}}:\left(\alpha,\beta\right)\rightarrow M$ are regular curves covering segments of light rays, then $\mathbf{x}:\left(a,b\right)\times \left(\alpha,\beta\right) \rightarrow M$ is said to be a \emph{variation by light rays}\index{light rays!variation}\index{variation!by light rays}.

Moreover, a variation by light rays $\mathbf{x}$ is said to be a \emph{variation by light rays of $\gamma\in \mathcal{N}$}\index{variation!by light rays!of a light ray $\gamma$} if $\gamma$ is a longitudinal curve of $\mathbf{x}$.
\end{definition}

\begin{notation}
It is possible to identify a given segment of null geodesic $\gamma:\left(-\delta,\delta\right)\rightarrow M$, with a slight abuse in the notation, to the light ray in $\mathcal{N}$ defined by it.
So, if $\mathbf{x}=\mathbf{x}\left(s,t\right)$ is a variation by light rays, we can denote by $\gamma_s^{\mathbf{x}}\subset M$ the null pregeodesics of the variation and also by $\gamma_s^{\mathbf{x}}\in \mathcal{N}$ the light rays they define.
\end{notation}

Consider a geodesic curve $\mu\left(t\right)$ in a spacetime $\left(M,\mathbf{g}\right)$.
Given $J\in \mathfrak{X}_{\mu}$, we will abbreviate the notation $J'=\frac{D J}{dt}$ and $J''=\frac{D}{dt}\frac{D J}{dt}=\frac{D^2 J}{dt^2}$.
We can define the \emph{Jacobi equation}\index{equation!Jacobi}\index{Jacobi!equation} by
\begin{equation}\label{eq-Jacobi}
J''+R\left(J,\mu'\right)\mu' =0
\end{equation}
where $R$ is the Riemann tensor.
We will name the solutions of the equation \ref{eq-Jacobi} by \emph{Jacobi field}\index{field!Jacobi}\index{Jacobi!field} along $\mu$.
So, the set of Jacobi fields along $\mu$ is then defined by
\begin{equation}\label{campos-Jacobi}
\mathcal{J}\left(\mu\right)=\left\{ J\in \mathfrak{X}_{\mu}: J''+R\left(J,\mu'\right)\mu' =0 \right\}
\end{equation}

The linearity of $\frac{D}{dt}$ and $R$ provides a vector space structure to $\mathcal{J}\left(\mu \right)$ hence $\mathcal{J}\left( \mu \right) $ is a vector subspace of $\mathfrak{X}_{\mu}$.

The relation between geodesic variations and Jacobi fields is expounded in next lemma.

\begin{lemma}\label{lema00351}
If $\mathbf{x}:\left(-\epsilon, \epsilon\right)\times \left(-\delta, \delta\right)\rightarrow M$ is a geodesic variation of a geodesic $\gamma$, then the initial field $V^{\mathbf{x}}$ is a Jacobi field along $\gamma$.
\end{lemma}

\begin{proof}
See \cite[Lem. 8.3]{On83}.
\hfill$\Box$\bigskip

\end{proof}

A Jacobi field along a geodesic $\gamma$ is fully defined by its initial values at any point of $\gamma$ as lemma \ref{lema00352} claims, and moreover it also implies that the vector space $\mathcal{J}\left( \mu \right)$ is isomorphic to $T_{p}M\times T_{p}M$ therefore $\dim \left(\mathcal{J}\left( \gamma \right) \right) =2\dim \left( M\right) =2m$.

\begin{lemma}\label{lema00352}
Let $\gamma $ be a geodesic in $M$ such that $\gamma \left( 0\right) =p$ and $u,v\in T_{p}M$.
Then there exists a only Jacobi field $J$ along $\gamma$ such that $J\left( 0\right) =u$ and $\frac{DJ}{dt}\left( 0\right) =v$.
\end{lemma}

\begin{proof}
See \cite[Lem. 8.5]{On83}.
\hfill$\Box$\bigskip

\end{proof}

Next lemma characterizes the Jacobi fields of a particular type of variation. This type will be the general case for the variations by light rays studied below.

\begin{lemma}
\label{lemmaDC92} Let $M$ be a spacetime, $\gamma:\left(-\delta,\delta\right)\rightarrow M$ a geodesic segment, $\lambda :\left( -\epsilon ,\epsilon\right)\rightarrow M$ a curve verifying $\lambda \left( 0\right)=\gamma\left( 0\right)$, and $W\left( s\right) $ a vector field along $\lambda $ such that $W\left( 0\right) =\gamma ^{\prime }\left( 0\right) $.
Then the Jacobi field $J$ along $\gamma$ defined by the geodesic variation
\begin{equation*}
\mathbf{x}\left( s,t\right) =\mathrm{exp}_{\lambda \left( s\right)
}\left(tW\left( s\right) \right)
\end{equation*}%
verifies that
\begin{equation*}
J\left(0\right) = \lambda'\left(0\right) \, , \qquad 
\frac{DJ}{dt}\left( 0\right)=\frac{DW}{ds}\left( 0\right)
\end{equation*}
\end{lemma}

\begin{proof}
First, the vector $\frac{\partial\mathbf{x}}{\partial s}\left(0,0\right)$ is the tangent vector of the curve $\mathbf{x}\left(s,0\right)$ in $s=0$, and since $\mathbf{x}\left(s,0\right)=\mathrm{exp}_{\lambda\left(s\right)}\left(0\cdot W\left(s\right)\right)= \mathrm{exp}_{\lambda\left(s\right)}\left(0\right)=
\lambda\left(s\right)$, then we have
\begin{equation*}
J\left(0\right)=\frac{\partial\mathbf{x}}{\partial s}\left(0,0\right)=\frac{d\lambda}{ds}%
\left(0\right)=\lambda ^{\prime }\left(0\right)
\end{equation*}
On the other hand, $\frac{D}{ds}\frac{\partial\mathbf{x}}{\partial t}\left(0,0\right)$ is the covariant derivative of the vector field $\frac{\partial\mathbf{x}}{\partial t}\left(s,0\right)=W\left(s\right)$ for $s=0$ along the curve  $\mathbf{x}\left(s,0\right)=\lambda\left(s\right)$.
Then
\begin{equation*}
\frac{DJ}{dt}\left(0\right)=\frac{D}{dt}\frac{\partial\mathbf{x}}{\partial s}\left(0,0\right)=\frac{D}{ds}\frac{\partial\mathbf{x}}{\partial t}\left(0,0\right)= \frac{DW}{ds}\left(0\right).
\end{equation*}
as required.
\hfill$\Box$\bigskip

\end{proof}

\begin{remark}\label{remark-Jacobi-init-val}
It can be observed that given a geodesic variation $\mathbf{x}=\mathbf{x}\left(s,t\right)$ such that $J$ is the corresponding Jacobi field at $s=0$, if we change the geodesic parameters such that $\overline{\mathbf{x}}\left(s,\tau\right)=\mathbf{x}\left(s,a\tau + b\right)$ for $a>0$ and $b\in\mathbb{R}$, then the Jacobi field $\overline{J}$ of $\overline{\mathbf{x}}$ at $s=0$ verify
$\overline{J}=J$. This implies that changing the geodesic parameter does not modify the Jacobi field as a geometric object.
\end{remark}

\begin{proposition}\label{prop00200}
Given a geodesic $\gamma$ in $\left(M,\mathbf{g}\right)$ and a Jacobi field $J\in\mathcal{J}\left(\gamma\right)$ along $\gamma$, then $\mathbf{g}\left( J\left(t\right),\gamma'\left(t\right) \right)=a+bt$ is verified.
\end{proposition}

\begin{proof}
It is trivial deriving $\mathbf{g}\left( J\left(t\right),\gamma'\left(t\right) \right)$ twice with respect to $t$.  
\hfill$\Box$\bigskip

\end{proof}


\begin{remark}\label{remark-lem00250}
Observe the following fact that we will need later. 
Given a differentiable curve $v:I\rightarrow TM$, the information contained in the tangent vector $v'\left(s_0\right)\in TTM$ coincides with the one in the covariant derivative of $v$ as vector field along its base curve in $M$. 
That is, from $v'\left(s_0\right)$ it is possible to determine the vectors $\alpha'\left(s_0\right)$ and the covariant derivative $\frac{Dv}{ds}\left( s_0\right)$ along $\alpha$, where $\alpha=\pi^{TM}_{M}\circ v$, and vice--versa.  
\end{remark}

It is possible to identify any tangent vector $\xi\in TTM$ with a Jacobi field along the geodesic $\gamma$ defined by the exponential of the vector $u=\pi^{TTM}_{TM} \left(\xi\right)\in TM$.

\begin{proposition}\label{prop00254}
Given a vector $u_0\in T_p M$ and consider the geodesic $\gamma_{u_0}$ defined by $\gamma_{u_0}\left(t\right)=\mathrm{exp}_p\left(tu_0\right)$.
Let $u:\left(-\delta,\delta\right)\rightarrow TM$ be a differentiable curve such that $u\left(0\right)=u_0$ and $u'\left(0\right)=\xi$.
If $J\in\mathcal{J}\left(\gamma_{u_0}\right)$ is the Jacobi field of the geodesic variation given by $\mathbf{x}\left(s,t\right)=\mathrm{exp}_{\alpha\left(s\right)}\left(t u\left(s\right)\right)$ where $\alpha = \pi_{M}^{TM} \circ u$, then the map
\[
\begin{tabular}{rrcl}
$\zeta:$ & $T_{u_0}TM$ & $\rightarrow$ & $\mathcal{J}\left(\gamma_{u_0}\right)$ \\
 &   $\xi$ & $\mapsto$ & $J$
\end{tabular}
\]
is a well--defined linear isomorphism.
\end{proposition}

\begin{proof}
Let $u_i:\left(-\delta_i,\delta_i\right)\rightarrow TM$ be differentiable curves such that $u_i\left(0\right)=u_0$ and $u_i'\left(0\right)=\xi$ for $i=1,2$, and consider the geodesic variations $\mathbf{x}_i\left(s_i,t\right)=\mathrm{exp}_{\alpha\left(s_i\right)}\left(t u_i\left(s_i\right)\right)$.
First, observe that for every $\xi\in T_{u_0}TM$, remark \ref{remark-lem00250} and lemma \ref{lemmaDC92} imply that $\xi$ defines unambiguously the initial values of the Jacobi field of $\mathbf{x}$ at $s=0$.
So $\zeta$ is well--defined.

The linearity of $\zeta$ is straightforward using remark \ref{remark-lem00250}.


Finally, if $\xi\in TM$ such that $\zeta\left(\xi\right)=0 $ then, in virtue of lemmas \ref{lemmaDC92} and remark \ref{remark-lem00250}, we have that $\xi = 0$.
This implies that $\zeta$ is injective, but since $\mathrm{dim}\left(\mathcal{J}\left(\gamma_{u_0}\right)\right) = \mathrm{dim}\left(T_{u_0}TM\right)=2m$ then we conclude that $\zeta$ is an isomorphism.
\hfill$\Box$\bigskip

\end{proof}

Now, we will focus on the variations by light rays and the Jacobi fields they define.
Fix a null geodesic $\gamma\in\mathcal{N}$ and assume that $\mathbf{x}\left( s,t\right) $ is a variation by light rays of $\gamma =\gamma_0^{\mathbf{x}}\in \mathcal{N}$ in such a way that $J\left( t\right) =V_{0}^{\mathbf{x}}\left(t\right) $ is the Jacobi field over $\gamma $ corresponding to the initial field of $\mathbf{x}$ and $\frac{\partial \mathbf{x}}{dt}\left(s,t\right)=\left(\gamma^{\mathbf{x}}_{s}\right)'\left(t\right)$.
Since $\mathbf{x}$ is a variation by light rays, then it provides that $\mathbf{g}\left( \left(\gamma^{\mathbf{x}}_{s}\right)'\left(t\right), \left(\gamma^{\mathbf{x}}_{s}\right)'\left(t\right) \right) =0$ for all
$\left( s,t\right) $ in the domain of $\mathbf{x}$, hence
\begin{equation*}
0=\left. \frac{\partial }{\partial s}\right|_{\left( 0,t\right)} \mathbf{g}\left( \left(\gamma^{\mathbf{x}}_{s}\right)'\left(t\right) ,\left(\gamma^{\mathbf{x}}_{s}\right)'\left(t\right) \right) =
2 \mathbf{g}\left( \left.\frac{D}{ds}\right|_{\left(0,t\right)}\frac{\partial \mathbf{x}}{dt}\left(s,t\right) ,\frac{\partial \mathbf{x}}{dt}\left(0,t\right) \right)=
\end{equation*}%
\begin{equation*}
=2 \mathbf{g}\left( \left.\frac{D}{dt}\right|_{\left(0,t\right)}\frac{\partial \mathbf{x}}{ds}\left(s,t\right) ,\frac{\partial \mathbf{x}}{dt}\left(0,t\right) \right)=\left. \frac{\partial }{\partial t}\right|_{\left( 0,t\right)} \mathbf{g}\left( V^{\mathbf{x}}_{s}\left(t\right) ,\left(\gamma^{\mathbf{x}}_{s}\right)'\left(t\right) \right)=
\end{equation*}%
\begin{equation*}
=\left. \frac{\partial }{\partial t}\right|_{t} \mathbf{g}\left( V^{\mathbf{x}}\left(t\right) ,\gamma'\left(t\right) \right) =
\left. \frac{\partial }{\partial t}\right|_{t} \mathbf{g}\left( J\left(t\right) ,\gamma'\left(t\right) \right)
\end{equation*}%
then the variations by light rays of a null geodesic $\gamma$ verify that their Jacobi fields $J$ fulfil
\begin{equation*}
\mathbf{g}\left( J\left( t\right) ,\gamma ^{\prime }\left( t\right) \right)=c
\end{equation*}%
with $c\in \mathbb{R}$ constant.
Then, we define the set of Jacobi fields of variations by light rays by
\begin{equation*}
\mathcal{J}_{L}\left( \gamma \right) =\left\{ J\in \mathcal{J}\left( \gamma
\right) :\mathbf{g}\left( J,\gamma^{\prime }\right)=c \text{ constant}%
\right\}
\end{equation*}%
Since $\mathbf{g}$ is bilinear then $\mathcal{J}_{L}\left( \gamma \right) $ is a vector subspace of $\mathcal{J}\left( \gamma \right) $ verifying $\mathrm{dim}\left(\mathcal{J}_{L}\left(\gamma\right)\right)=2\mathrm{dim}\left(M\right)-1=2m-1$.

Now, we define one--dimensional subspaces of $\mathcal{J}\left(\gamma\right)$ given by
$$
\widehat{\mathcal{J}}_{0}\left(\gamma\right) =\left\{  J\left(t\right)=bt\gamma'\left(t\right):b\in\mathbb{R} \right\} \, ,\qquad 
\widehat{\mathcal{J}}'_{0}\left(\gamma\right) =\left\{  J\left(t\right)=a\gamma'\left(t\right):a\in\mathbb{R} \right\} \, .
$$
It is trivial to see that $\widehat{\mathcal{J}}_{0}\left(\gamma\right)\subset \mathcal{J}_{L}\left(\gamma\right) $ and $\widehat{\mathcal{J}}'_{0}\left(\gamma\right) \subset \mathcal{J}_{L}\left(\gamma\right)$.

If $J\in \widehat{\mathcal{J}}_{0}\left(\gamma\right) \cap \widehat{\mathcal{J}}'_{0}\left(\gamma\right)$, then its initial values must verify
\[
\left\{
\begin{array}{l}
J\left(0\right)=0 \\
J'\left(0\right)=b\gamma'\left(0\right)
\end{array}
\right.
\hspace{1cm}
\mathrm{and}
\hspace{1cm}
\left\{
\begin{array}{l}
J\left(0\right)=a\gamma'\left(0\right) \\
J'\left(0\right)=0
\end{array}
\right.
\]
then $a=b=0$ and therefore $\widehat{\mathcal{J}}_{0}\left(\gamma\right) \cap \widehat{\mathcal{J}}'_{0}\left(\gamma\right)= \left\{0\right\}$.
So, we can define the direct product
\[
\mathcal{J}_{0}\left(\gamma\right)= \widehat{\mathcal{J}}_{0}\left(\gamma\right) \oplus \widehat{\mathcal{J}}'_{0}\left(\gamma\right)= \left\{  J\left(t\right)=\left(a+bt\right)\gamma'\left(t\right):a,b\in\mathbb{R} \right\}
\]
being the vector subspace of Jacobi fields proportional to $\gamma'$ and verifying $\mathrm{dim}\left(\mathcal{J}_{0}\left(\gamma\right)\right)=2$.

Now, we can define the quotient vector space
\[
\mathcal{L}\left(\gamma\right)= \mathcal{J}_{L}\left(\gamma\right) / \mathcal{J}_{0}\left(\gamma\right) = \left\{  \left[J\right]: K\in \left[J\right] \Leftrightarrow K=J + J_0 \text{ such that } J_0\in \mathcal{J}_{0}\left(\gamma\right) \right\}
\]
whose dimension is  $\mathrm{dim}\left(\mathcal{L}\left(\gamma\right)\right)=\mathrm{dim}\left(\mathcal{J}_{L}\left(\gamma\right)\right) - \mathrm{dim}\left(\mathcal{J}_{0}\left(\gamma\right)\right)=2\mathrm{dim}\left(M\right)-3$.
The elements of $\mathcal{L}\left(\gamma\right)$ will be denoted by $\left[J\right]\equiv J\left(\mathrm{mod}\gamma'\right)$ and we will say that $K=J\left(\mathrm{mod}\gamma'\right)$ when $\left[K\right]=\left[J\right]$.

Next lemma claims that there exist a change of parameter such that any variation by light rays can be transformed in a geodesic variation by light rays.
So, lemma \ref{lema00351} can be used.

\begin{lemma}\label{lemma-g-geodesic}
Let $\mathbf{x}=\mathbf{x}\left(s,t\right)$ be a variation by light rays in $\left(M,\mathcal{C}\right)$ such that $\gamma_s\left(t\right)=\mathbf{x}\left(s,t\right)$ defines its light rays.
Fixed any metric $\mathbf{g}\in \mathcal{C}$ then there exists a differentiable function $h=h\left(s,\tau\right)$ such that the light rays parametrized as $\overline{\gamma}_s=\gamma_s\left(h\left(s,\tau\right)\right)$ are null geodesics related to $\mathbf{g}$.
\end{lemma}

\begin{proof}
Since each $\gamma_s$ is a segment of light ray then $\gamma_s=\gamma_s\left(t\right)$ is a pregeodesic related to $\mathbf{g}$.
Hence
\[
\frac{D\gamma_s '\left(t\right)}{dt}=\frac{D}{dt}\frac{\partial \mathbf{x}}{\partial t}\left(s,t\right)=f\left(s,t\right) \gamma'_s\left(t\right)
\]
where $f$ is differentiable and $\frac{D}{dt}$ denotes the covariant derivative related to $\mathbf{g}$ along $\gamma_s\left(t\right)$.

It is enough to check that the function $h\left(s,\tau\right)=h_s\left(\tau\right)$ defined by 
\begin{equation}\label{h-inversa}
h^{-1}_{s}\left(t\right)  = \int_{0}^{t} e^{ \int_{0}^{x} f\left(s,y\right)dy}dx
\end{equation}
makes of $\overline{\gamma}_s=\gamma_s \circ h$ a null geodesic with respect to the metric $\mathbf{g}$. 
\hfill$\Box$\bigskip

\end{proof}

Lemma \ref{lem00200} shows that any differentiable curve $\Gamma\subset \mathcal{N}$ defines a variation by light rays $\mathbf{x}$ such that the longitudinal curves of $\mathbf{x}$ corresponds to points in $\Gamma$.
This variation is not unique by construction.

\begin{lemma}\label{lem00200}
Given a differentiable curve $\Gamma:I\rightarrow \mathcal{N}$ such that $0\in I$ and $\Gamma\left(s\right)=\gamma_s \subset M$, then there exists a variation by light rays $\mathbf{x}:\left(-\epsilon, \epsilon\right)\times \left(-\delta, \delta\right)\rightarrow M$ verifying
\[
\mathbf{x}\left(s,t\right)=\gamma_s\left(t\right)
\]
for all $\left(s,t\right)\in \left(-\epsilon, \epsilon\right)\times \left(-\delta, \delta\right)$.
Moreover, the variation $\mathbf{x}$ can be written as 
\[
\mathbf{x}\left(s,t\right)=\mathrm{exp}_{\pi^{\mathbb{N^{+}}}_{M}\left(v\left(s\right)\right)}\left(t v\left(s\right)\right)
\]
where $v:\left(-\epsilon,\epsilon\right)\rightarrow \mathbb{N}^{+}\left(C\right)$ is a differentiable curve.
\end{lemma}

\begin{proof}
Consider the restriction $\pi=\left.\pi^{\mathbb{N}^{+}}_{\mathbb{PN}}\right|_{\mathbb{N}^{+}\left(C\right)}:\mathbb{N}^{+}\left(C\right)\rightarrow \mathbb{PN}\left(C\right)$ and the diffeomorphism $\sigma:\mathbb{PN}\left(C\right)\rightarrow \mathcal{U}$ in the diagram \ref{diagrama002}, where $\mathcal{U}\subset \mathcal{N}$ and $V\subset M$ are open such that $V$ is globally hyperbolic and $C\subset V$ is a Cauchy surface of $V$ and moreover $\gamma_0\in \mathcal{U}$, in such a way the following diagram arise
\begin{equation}\label{diagrama004}
\begin{tikzpicture}[every node/.style={midway}]
\matrix[column sep={6em,between origins},
        row sep={2em}] at (0,0)
{ \node(PN1)   {$\mathbb{PN}\left( C\right)$}  ; & \node(N) {$\mathcal{U}$}; \\
  \node(N2) {$\mathbb{N}^{+}\left( C\right)$};                   \\};
\draw[->] (PN1) -- (N) node[anchor=south]  {$\sigma $};
\draw[->] (N2) -- (N) node[anchor=north]  {\hspace{5mm}  $\sigma \circ\pi$};
\draw[<-] (PN1)   -- (N2) node[anchor=east] {$\pi$};
\end{tikzpicture}
\end{equation}
Also consider the canonical projection $\pi^{\mathbb{N}^{+}}_{M}:\mathbb{N}^{+}\rightarrow M$ and the exponential map $\mathrm{exp}:\left(-\delta,\delta\right)\times \mathbb{N}^{+}\rightarrow M$ defined by $\mathrm{exp}\left(t,v\right)=\mathrm{exp}_{\pi_{M}^{\mathbb{N}^{+}}\left(v\right)}\left(tv\right)$.
Fix $\epsilon>0$ such that $\Gamma\left(s\right)\in \mathcal{U}$ for all $s\in \left(-\epsilon, \epsilon\right)$ and let $z:\mathbb{PN}\left(C\right)\rightarrow \mathbb{N}^{+}\left(C\right)$ be a section of $\pi$ that, without restriction of generality, can be considered a global section due to the locality of $\pi$.
Naming $v\left(s\right)=z\circ \sigma^{-1}\circ \Gamma\left(s\right)$ for $s\in \left(-\epsilon,\epsilon\right)$, then we can define a variation $\mathbf{x}:\left(-\epsilon, \epsilon\right)\times \left(-\delta, \delta\right)\rightarrow M$ by $\mathbf{x}\left(s,t\right)=\mathrm{exp}\left(t,v\left(s\right)\right)=\mathrm{exp}_{\pi^{\mathbb{N^{+}}}_{M}\left(v\left(s\right)\right)}\left(t v\left(s\right)\right)$.
By construction as a composition of differentiable maps, $\mathbf{x}$ is differentiable.
Moreover, since $ v\left(s\right)$ is the initial vector of the geodesic $\gamma^{\mathbf{x}}_s$ defined by $\mathbf{x}\left(s,t\right)=\gamma^{\mathbf{x}}_s\left(t\right)$, then
\[
\gamma^{\mathbf{x}}_s = \sigma \circ \pi\left(v\left(s\right)\right) =\sigma \circ \pi\circ z\circ \sigma^{-1}\circ \Gamma\left(s\right)=\sigma \circ \sigma^{-1}\circ \Gamma\left(s\right)=\Gamma\left(s\right)
\]
for all $s\in\left(-\epsilon,\epsilon\right)$, and the lemma follows.
\hfill$\Box$\bigskip

\end{proof}

\begin{lemma}\label{lem00210}
Given a variation $\mathbf{x}:\left(-\epsilon, \epsilon\right)\times \left(-\delta, \delta\right)\rightarrow M$ by light rays such that $\mathbf{x}\left(s,t\right)=\gamma_{s}^{\mathbf{x}}\left(t\right)$, then the curve $\Gamma^{\mathbf{x}}:I\rightarrow \mathcal{N}$ verifying $\Gamma^{\mathbf{x}}\left(s\right)=\gamma_{s}^{\mathbf{x}}$ is differentiable.
\end{lemma}

\begin{proof}
Let $\mathbf{x}:\left(-\epsilon, \epsilon\right)\times \left(-\delta, \delta\right) \rightarrow M$ be a variation by light rays such that  $\gamma_{s}^{\mathbf{x}} \left(t\right)=\mathbf{x}\left(s,t\right)$.
Then the curve
\[
\lambda\left(s \right)= d\mathbf{x}_{\left(s,0\right)}\left(\frac{\partial}{\partial t}\right)_{\left(s,0\right)}\in \mathbb{N}^{+}
\]
is clearly differentiable.
If $p_{\mathbb{N}^{+}}:\mathbb{N}^{+} \rightarrow \mathcal{N}$ is the submersion of proposition \ref{prop00150}, then $p_{\mathbb{N}^{+}}\circ \lambda: I\rightarrow \mathcal{N}$ is differentiable in $\mathcal{N}$, by composition of differentiable maps.
Since
\[
p_{\mathbb{N}^{+}}\circ\lambda\left(s\right)=
p_{\mathbb{N}^{+}}\left(\left(\gamma_{s}^{\mathbf{x}}\right)'\left(0\right)\right)=\gamma_{s}^{\mathbf{x}}=\Gamma^{\mathbf{x}}\left(s\right).
\]
then $\Gamma^{\mathbf{x}}$ is also differentiable.
\hfill$\Box$\bigskip

\end{proof}

Let us adopt the notation used in lemma \ref{lem00210} and call $\Gamma^{\mathbf{x}}$ the curve in $\mathcal{N}$ defined by the variation $\mathbf{x}$ by light rays such that if $\mathbf{x}\left(s,t\right)=\gamma_{s}^{\mathbf{x}}\left(t\right)$ then $\Gamma^{\mathbf{x}}\left(s\right)=\gamma_{s}^{\mathbf{x}}\in \mathcal{N}$.

Although the variations defined in lemma \ref{lem00200} are not unique, lemma \ref{lem00215} shows that all they define the same Jacobi field except by a term in the direction of $\gamma'$.

\begin{lemma}\label{lem00215}
Let $\overline{\mathbf{x}}:I\times \overline{H}\rightarrow M$ and $\mathbf{x}:I\times H\rightarrow M$ be variations by light rays such that $\Gamma^{\overline{\mathbf{x}}}\left(s\right)=\gamma^{\overline{\mathbf{x}}}_{s}$ and $\Gamma^{\mathbf{x}}\left(s\right)=\gamma^{\mathbf{x}}_{s}$ with $\gamma^{\overline{\mathbf{x}}}_{0}=\gamma^{\mathbf{x}}_{0}=\gamma\in \mathcal{N}$.
Let us denote by $\overline{J}$ and $J$ the Jacobi fields over $\gamma$ of $\overline{\mathbf{x}}$ and $\mathbf{x}$ respectively.
If $\Gamma^{\overline{\mathbf{x}}}=\Gamma^{\mathbf{x}}$ then $\overline{J}=J\left(\mathrm{mod}\gamma'\right)$.
\end{lemma}

\begin{proof}
We have that $\overline{\mathbf{x}}\left(s,t\right)=\gamma^{\overline{\mathbf{x}}}_{s}\left(t\right)$ and $\mathbf{x}\left(s,\tau\right)=\gamma^{\mathbf{x}}_{s}\left(\tau\right)$. By lemma \ref{lemma-g-geodesic}, we can assume without any lack of generality, that $\gamma^{\overline{\mathbf{x}}}_{s}$ are null geodesics for the metric $\mathbf{g}\in\mathcal{C}$ giving new parameters if necessary.
If $\Gamma^{\overline{\mathbf{x}}}=\Gamma^{\mathbf{x}}$ then $\gamma^{\overline{\mathbf{x}}}_{s}=\gamma^{\mathbf{x}}_{s}$ for all $s\in I$.
Then there exist a differentiable function $h_s\left(t\right)=h\left(s,t\right)$ such that  $\overline{\mathbf{x}}\left(s,t\right)=\mathbf{x}\left(s,h\left(s,t\right)\right)$.
Hence we have that
\[
\frac{\partial\overline{\mathbf{x}}\left(s,t\right)}{\partial s}= \frac{\partial\mathbf{x}\left(s,h\left(s,t\right)\right)}{\partial s} + \frac{\partial h\left(s,t\right)}{\partial s}  \cdot \frac{\partial\mathbf{x}\left(s,h\left(s,t\right)\right)}{\partial \tau}
\]
then if $s=0$
\[
\overline{J}\left(t\right)= J\left(h\left(0,t\right)\right) + \frac{\partial h}{\partial s}\left(0,t\right) \cdot \gamma'\left(t\right)
\]
therefore $\overline{J}=J\left(\mathrm{mod}\gamma'\right)$.
\hfill$\Box$\bigskip

\end{proof}

We can wonder how a Jacobi field changes when another metric of the same conformal class is considered in $M$.
The following result shows it with a proof similar to the one of lemma \ref{lem00215}.

\begin{lemma}\label{JJbar}
Let $\mathbf{x}:I\times H \rightarrow M$ be a variation by light rays of $\gamma=\mathbf{x}\left(0,\cdot\right)$.
If $J\in \mathcal{J}_L\left(\gamma\right)$ is the Jacobi field of $\mathbf{x}$ along $\gamma$ related to the metric $\mathbf{g}\in\mathcal{C}$, then the Jacobi field $\overline{J}$ of $\mathbf{x}$ along $\gamma$ related to another metric $\overline{\mathbf{g}}\in\mathcal{C}$ verifies
\[
\overline{J}= J \left(\mathrm{mod}\gamma'\right).
\]
\end{lemma}

\begin{proof}
Let $\mathbf{x}:I\times H \rightarrow M$ be a variation by light rays of $\gamma$ where $\mathbf{x}\left(s,t\right)=\gamma_s\left(t\right)$ with $\gamma=\gamma_0$.
By lemma \ref{lemma-g-geodesic}, we can assume that $\gamma_s$ is null geodesic related to the metric $\mathbf{g}$ and there exists changes of parameters $h_s:\overline {H}\rightarrow H$ such that $\overline{\gamma_s}\left(\tau\right)=\gamma_s\left(h_s\left(\tau\right)\right)$ are null geodesics related to $\overline{\mathbf{g}}\in \mathcal{C}$ for all $s\in I$ and where $h\left(s,\tau\right)=h_s\left(\tau\right)$ is a differentiable function.
So, we consider that $J\in \mathcal{J}_L\left(\gamma\right)$ is the Jacobi field of $\mathbf{x}$ along $\gamma$ and $\overline{J}\in \overline{\mathcal{J}}_L\left(\gamma\right)$ is the one of $\overline{\mathbf{x}}$.
Then $\overline{\mathbf{x}}\left(s,\tau\right)=\mathbf{x}\left(s,h_s\left(\tau\right)\right)$ and we have that
\begin{align*}
\overline{J}\left(\tau\right)&= \left.\frac{\partial \overline{\mathbf{x}}\left(s,\tau\right)}{\partial s}\right|_{\left(0,\tau\right)} = \left.\frac{\partial \mathbf{x}\left(s,h_s\left(\tau\right)\right)}{\partial s}\right|_{\left(0,\tau\right)} = \\
&= \frac{\partial \mathbf{x}}{\partial s}\left(0,h_0\left(\tau\right)\right) + \frac{\partial h}{\partial s}\left(0,\tau\right)\frac{\partial \mathbf{x}}{\partial t}\left(0,h_0\left(\tau\right)\right) = J\left(h_0\left(\tau\right)\right) + \frac{\partial h}{\partial s}\left(0,\tau\right)\gamma'\left(h_0\left(\tau\right)\right)
\end{align*}
therefore $\overline{J}= J \left(\mathrm{mod}\gamma'\right)$.
\hfill$\Box$\bigskip

\end{proof}

\begin{lemma}\label{lemma270}
Given two variations by light rays $\mathbf{x}:I\times H\rightarrow M$ and $\overline{\mathbf{x}}:\overline{I}\times \overline{H}\rightarrow M$ such that $\Gamma^{\mathbf{x}}\left(0\right)=\Gamma^{\overline{\mathbf{x}}}\left(0\right)=\gamma$. Let us denote by $J$ and $\overline{J}$ their corresponding Jacobi fields at $0\in I$ and $0\in \overline{I}$ of $\mathbf{x}$ and $\overline{\mathbf{x}}$ respectively.
If $\left(\Gamma^{\mathbf{x}}\right)'\left(0\right)=\left(\Gamma^{\overline{\mathbf{x}}}\right)'\left(0\right)$ then $J=\overline{J}\left(\mathrm{mod}\gamma'\right)$.
\end{lemma}

\begin{proof}
Due to we want to compare the Jacobi fields $J$ and $\overline{J}$ on $\gamma$, we can assume without any lack of generality that $\mathbf{x}$ as well as $\overline{\mathbf{x}}$ provide the same geodesic parameter for $\gamma$, then by lemmas \ref{lem00200} and \ref{lem00215}, we can consider that $\mathbf{x}\left(s,t\right)=\mathrm{exp}_{\alpha\left(s\right)}\left(t u\left(s\right)\right)$ and $\overline{\mathbf{x}}\left(r,t\right)=\mathrm{exp}_{\overline{\alpha}\left(r\right)}\left(t \overline{u}\left(r\right)\right)$ where $u=u\left(0\right)=\overline{u}\left(0\right)$ and also $p=\alpha\left(0\right)=\overline{\alpha}\left(0\right)$.

Moreover, we can assume the diagram \ref{diagrama004} holds.

Since $\left(\Gamma^{\mathbf{x}}\right)'\left(0\right)=\left(\Gamma^{\overline{\mathbf{x}}}\right)'\left(0\right)$ then we have
\[
d\sigma_{\left[v\left(0\right)\right]}\circ d\pi_{v\left(0\right)}\left(v'\left(0\right)\right)=d\sigma_{\left[\overline{v}\left(0\right)\right]}\circ d\pi_{\overline{v}\left(0\right)}\left(\overline{v}'\left(0\right)\right) \Leftrightarrow d\pi_{v\left(0\right)}\left(v'\left(0\right)\right)= d\pi_{\overline{v}\left(0\right)}\left(\overline{v}'\left(0\right)\right)
\]
Observe that $\left[v\left(0\right)\right]=\left[\overline{v}\left(0\right)\right]$ and thus, $d\pi_{v\left(0\right)}=d\pi_{\overline{v}\left(0\right)}$, and its kernel is the subspace generated by the tangent vector at $s=0$ of the curve $c\left(s\right)=e^s v\left(0\right)$, hence
\begin{equation}\label{expresion-vprima}
v'\left(0\right)= \overline{v}'\left(0\right) + \mu c'\left(0\right)
\end{equation}
with $\mu\in\mathbb{R}$.
By remark \ref{remark-lem00250}, we have that
\[
\left\{
\begin{tabular}{l}
$\alpha'\left(0\right)=\overline{\alpha}'\left(0\right)$ \\
$\frac{Du}{ds}\left( 0\right) = \frac{D\overline{u}}{dr}\left( 0\right) + \mu\frac{Dc}{ds}\left( 0\right)$
\end{tabular}
\right.
\Rightarrow
\left\{
\begin{tabular}{l}
$\alpha'\left(0\right)=\overline{\alpha}'\left(0\right)$ \\
$\frac{Du}{ds}\left( 0\right) = \frac{D\overline{u}}{dr}\left( 0\right) + \mu\gamma'\left( 0\right)$
\end{tabular}
\right.
\]
therefore we conclude that $J=\overline{J}\left(\mathrm{mod}\gamma'\right)$.
\hfill$\Box$\bigskip

\end{proof}

The differentiable structure of $\mathcal{N}$ has been built in section~\ref{sec:seccion-estruc-dif} from the one in $\mathbb{PN}\left(C\right)$ where $C$ is a local spacelike Cauchy surface.
So, we will identify the tangent space $T_{\gamma}\mathcal{N}$ with some quotient space of $\mathcal{J}_L\left(\gamma\right)$ via a tangent space of $\mathbb{PN}\left(C\right)$.

\begin{proposition}\label{prop00399}
Given $\xi\in T_{\gamma_{u_0} }\mathcal{N}$ such that $\Gamma'\left(0\right)=\xi$ for some curve $\Gamma \subset \mathcal{N}$. Let $\mathbf{x}=\mathbf{x}\left( s,t\right) $ be a variation by light rays of $\gamma_{u_0}$ verifying that $\Gamma^{\mathbf{x}}=\Gamma $ such that $J\in \mathcal{L}\left( \gamma_{u_0} \right)$ is the Jacobi field over $\gamma_{u_0}$ of $\mathbf{x}$.
If $\zeta :T_{\gamma_{u_0} }\mathcal{N} \rightarrow \mathcal{L}\left( \gamma_{u_0} \right)$ is the map defined by
\[
\overline{\zeta}\left(\xi\right)=J\left(\mathrm{mod}\gamma'_{u_0}\right)
\]
then $\overline{\zeta}$ is well--defined and a linear isomorphism.
\end{proposition}

\begin{proof}
By lemma \ref{lemma270}, $\overline{\zeta}$ is well--defined.

We have seen in section~\ref{sec:seccion-estruc-dif} that for a globally hyperbolic open set $V\subset M$ such that $C\subset V$ is a smooth local spacelike Cauchy surface, the diagram \ref{cadena-espacios} given by
\[
\mathcal{N} \supset \mathcal{U} \simeq \mathbb{PN}\left(C\right) \simeq \Omega^{X}\left(C\right) \hookrightarrow \mathbb{N}^{+}\left(C\right) \hookrightarrow \mathbb{N}^{+} \hookrightarrow TM
\]
holds.
Proposition \ref{prop00254} shows that $\zeta:T_u TM \rightarrow \mathcal{J}\left(\gamma_u\right)$ is a linear isomorphism for any $u\in TM$.

The idea of the proof is the following. We will restrict $\zeta$ from $T_u TM$ up to $T_{\left[u\right]}\mathbb{PN}\left(C\right)$ step by step, identifying the corresponding subspace of $\mathcal{J}\left(\gamma_u\right)$ image of the map. In the first step we obtain that $T_{u_0}\mathbb{N}^{+} \rightarrow \mathcal{J}_L\left(\gamma_{u_0}\right)$ is a isomorphism. 
In the next step we obtain the isomorphism  $T_{u_0}\mathbb{N}^{+}\left(C\right) \rightarrow S$ where $S\subset \mathcal{J}_L\left(\gamma_{u_0}\right)$ is of the same codimension and transverse (that is, linearly independent) to the vector subspace $\widehat{\mathcal{J}}'_0\left(\gamma_{u_0}\right)$, then $S$ is isomorphic to $\mathcal{J}_L\left(\gamma_{u_0}\right) / \widehat{\mathcal{J}}'_0 \left(\gamma_{u_0}\right)$. 
Next, we consider the isomorphism $T_{\left[u_0\right]}\mathbb{PN}\left(C\right) \rightarrow S / \left(S\cap \widehat{\mathcal{J}}_0\left(\gamma_{u_0}\right)\right)$, but since $\mathcal{U} \simeq \mathbb{PN}\left(C\right)$, then 
\[
T_{\gamma_{u_0}}\mathcal{N} \rightarrow  S / \left(S\cap \widehat{\mathcal{J}}_0\left(\gamma_{u_0}\right)\right)
\]
is an isomorphism.

Recall that we have denoted $\mathcal{J}_0\left(\gamma_{u_0}\right) =\widehat{\mathcal{J}}_0\left(\gamma_{u_0}\right) \oplus \widehat{\mathcal{J}}'_0\left(\gamma_{u_0}\right)$.
Observe that the linear map $q:S\rightarrow \mathcal{J}_L\left(\gamma_{u_0}\right) /\mathcal{J}_0\left(\gamma_{u_0}\right)$ defined by $q\left(J\right)=\left[J\right]$ verifies that
\[
q\left(J\right)=\left[0\right] \Leftrightarrow J\left(t\right)=\left(a+bt\right)\gamma'_{u_0}\left(t\right) \Leftrightarrow J\in S\cap \widehat{\mathcal{J}}_0\left(\gamma_{u_0}\right)
\]
then $S / \left(S\cap \widehat{\mathcal{J}}_0\left(\gamma_{u_0}\right)\right)$ is isomorphic to $\mathcal{L}\left(\gamma_{u_0}\right)=\mathcal{J}_L\left(\gamma_{u_0}\right) /\mathcal{J}_0\left(\gamma_{u_0}\right)$.
This shows that the map $\overline{\zeta}: T_{\gamma_{u_0}}\mathcal{N} \rightarrow \mathcal{L}\left(\gamma_{u_0}\right)=\mathcal{J}_L\left(\gamma_{u_0}\right) /\mathcal{J}_0\left(\gamma_{u_0}\right)$, 
$\xi \mapsto \left[J\right]$, is a linear isomorphism and the proof is complete.
\hfill$\Box$\bigskip

\end{proof}

Proposition \ref{prop00399} allows to see the vectors of the tangent space $T_{\gamma}\mathcal{N}$ as Jacobi fields of variations by light rays.
We will use, from now on, this characterization when working with tangent vectors of $\mathcal{N}$.

By propositions \ref{prop00100} and proposition \ref{prop00399}, it is clear that the characterization of $T_{\gamma}\mathcal{N}$ as $\mathcal{L}\left(\gamma\right)$ does not depends on the representative of the conformal class $\mathcal{C}$.


\section{The canonical contact structure in $\mathcal{N}$}\label{sec:contact}

In this section, the canonical contact structure on $\mathcal{N}$ will be discussed.
Such contact struture is inherited from the kernel of the \emph{canonical 1--form} of $T^*M$ and it will be
described by passing the distribution of hyperplanes to $TM$ before pushing it down to $\mathcal{N}$ in virtue of the inclusions of eq. \ref{cadena-espacios}.
The basic elements of symplectic and contact geometry can be consulted in \cite{Ab87}, \cite{Ar89} and \cite{LM87}. 


\subsection{Elements of symplectic geometry in $T^{*}M$}\label{sec:T*M}


Consider a differentiable manifold $M$.  Its cotangent bundle $\pi \colon T^{*}M\rightarrow M$ carries a canonical 1--form $\theta$ defined pointwise at every $\alpha\in T^{*}M$ by $\theta_{\alpha} = \left(d\pi_{\alpha}\right)^{*}\alpha$.
Consequently we have
\begin{equation}\label{tautform}
\theta_{\alpha}\left(\xi\right) = \left(\left(d\pi_{\alpha}\right)^{*}\alpha\right) \left(\xi\right) = \alpha \left(\left(d\pi_{\alpha}\right)\xi\right)
\end{equation}
for $\xi\in T_{\alpha}\left(T^{*}M\right)$.
In local canonical bundle coordinates $(x^k,p_k)$, we can write
\begin{equation}\label{tautform-coord}
\theta = \sum_{k=1}^{m} p_k dx^k \, .
\end{equation}
The 2--form $\omega =-d\theta$,
defines a symplectic 2--form in $T^{*}M$, that in the previous local coordinates takes the form $\omega = \sum_{k=1}^{m}dx^k \wedge dp_{k}$.

Now, we want to construct $\mathcal{N}$ again, but this time starting from the contangent bundle $T^{*}M$.
Consider again the natural identification provided by the metric $\mathbf{g}$: $\widehat{\mathbf{g}}: TM \rightarrow T^{*}M$,  $\xi \mapsto \mathbf{g}\left(\xi,\cdot\right)$,
and denote by $\mathbb{N}^{+*}$ the image of the restriction of $\widehat{\mathbf{g}}$ to $\mathbb{N}^{+}$, that is
\[
\mathbb{N}^{+*}=\widehat{\mathbf{g}}\left(\mathbb{N}^{+}\right)=\left\{ \alpha=\widehat{\mathbf{g}}\left(\xi\right)\in T^{*}M: \xi\in \mathbb{N}^{+}\right\}
\]

In an analogous manner as in Sect. \ref{sec:seccion-estruc-dif}, define the \emph{Euler field} $\mathcal{E}\in \mathfrak{X}\left(T^{*}M\right)$\index{field!Euler!in $T^{*}M$}\index{Euler!field in $T^{*}M$} by
\begin{equation*}
\mathcal{E}\left( \alpha\right) =dc\left(\frac{\partial }{\partial t}\right) \left( 0\right) \, ,
\end{equation*}%
where $\alpha\in T_{p}^{*}M$ and $c:\mathbb{R}\rightarrow T_{p}^{*}M$ verifies that $c\left( t\right) =e^t\alpha$.
The curve $c$ is an integral curve of $\mathcal{E}$ because
\begin{equation*}
c^{\prime }\left( t\right) =dc\left( \frac{\partial }{\partial t}\right)\left( t\right) =\mathcal{E}\left( c\left( t\right) \right) \, .
\end{equation*}%
In the previous coordinates, $\mathcal{E}$ can be written as $\mathcal{E}=p_k \partial /\partial p_k$.
So, for every $\alpha\in \mathbb{N}^{+*}$ the integral curve $c\left(t\right)=e^t\alpha$ is contained in $\mathbb{N}^{+*}$, therefore $\mathcal{E}$ is tangent to $\mathbb{N}^{+*}$.

Moreover, if $\omega $ is the symplectic 2--form of $T^{*}M$ it is trivial to see 
\begin{equation}\label{omega-theta-2}
\mathcal{L}_{\mathcal{E}}\omega = i_{\mathcal{E}}d\omega + d\left(i_{\mathcal{E}}\omega\right) = d\left(-\theta\right)=-d\theta = \omega \, ,
\end{equation}
therefore $\mathcal{E}$ is a Liouville vector field.
In fact, $\mathcal{E}$ sometimes is called the \emph{Liouville} or \emph{Euler--Liouville vector field}\index{Euler--Liouville!vector field}\index{vector field!Euler--Liouville}.

Consider now the Hamiltonian function (again just the kinetic energy) defined by
$H: T^{*}M \rightarrow \mathbb{R}$, $\alpha \mapsto \frac{1}{2}\mathbf{g}\left(\widehat{\mathbf{g}}^{-1}\left(\alpha\right),\widehat{\mathbf{g}}^{-1}\left(\alpha\right)\right)$, defining the Hamiltonian vector field:
\[
X_{H} = g^{ki}p_i \frac{\partial}{\partial x^{k}} - \frac{1}{2}\frac{\partial g^{ij}}{\partial x^{k}}p_i p_j \frac{\partial}{\partial p_{i}}
\]

\begin{lemma}\label{lemma-spray-hamilton}
Let $X_{\mathbf{g}}, \Delta \in \mathfrak{X}\left(TM\right)$ be the the geodesic spray and Euler field of $TM$ and $X_{H},\mathcal{E}\in \mathfrak{X}\left(T^{*}M\right)$ the Hamiltonian vector field and Euler field of $T^{*}M$ respectively.   Then we have that $\widehat{\mathbf{g}}_{*}\left(\Delta\right)=\mathcal{E}$ and $\widehat{\mathbf{g}}_{*}\left(X_{\mathbf{g}}\right)=X_H$.
\end{lemma}

\begin{proof}
If we take any $\xi\in T^{*}M$ and $\alpha=\widehat{\mathbf{g}}\left(\xi\right)$, then the integral curve  $c\left(t\right)=e^t\xi$ of Euler field $\Delta$ in $TM$ is transformed by $\widehat{\mathbf{g}}$ as
\[
\widehat{\mathbf{g}}\left(c\left(t\right)\right)=\mathbf{g}\left(c\left(t\right),\cdot\right)=
\mathbf{g}\left(e^t\xi,\cdot\right)=e^t\mathbf{g}\left(\xi,\cdot\right)=e^t\widehat{\mathbf{g}}\left(\xi\right)=e^t\alpha\in T^{*}M
\]
being an integral curve of Euler field $\mathcal{E}$ in $T^{*}M$.
Then, for any $\xi\in T^{*}M$ we have that
\[
\widehat{\mathbf{g}}_{*}\left(\Delta\left(\xi\right)\right)=\mathcal{E}\left(\widehat{\mathbf{g}}\left(\xi\right)\right)
\]
is verified, therefore this implies $\widehat{\mathbf{g}}_{*}\left(\Delta\right)=\mathcal{E}$.

The second relation is obtained easily by taking the pull-back of the identity $i_{X_H}\omega = dH$ along the map $\hat{\mathbf{g}}$.
\hfill$\Box$\bigskip

\end{proof}

The following corollary is an immediate consequence of Lemma \ref{lemma-spray-hamilton} and the construction of $\mathcal{N}$ done in section~\ref{sec:seccion-estruc-dif}.

\begin{corollary}\label{corolario-N-T*M}
The space of light rays $\mathcal{N}$ of $M$ can be built by the quotient
\[
\mathcal{N} = \mathbb{N}^{+*} / \mathcal{D}^{*}
\]
where $\mathcal{D}^{*}$ is the distribution generated by the vector fields $\mathcal{E}$ and $X_H$, that is $\mathcal{D}^{*}=\mathrm{span}\left\{ \mathcal{E},X_H \right\}$.
\end{corollary}

Lemma \ref{lemma-spray-hamilton} also shows that the null geodesic defined by $\alpha\in\mathbb{N}^{*}$ coincides to the null geodesic defined by $v\in\mathbb{N}$ if and only if $\widehat{\mathbf{g}}\left(v\right)=\alpha$, because the first equation has to be verified.
Then we have the following commutative diagram:
\begin{equation}\label{diagrama006}
\begin{tikzpicture}[every node/.style={midway}]
\matrix[column sep={6em,between origins},
        row sep={2em}] at (0,0)
{ \node(N1)   {$\mathbb{N}^{*}$}  ; & \node(N) {$\mathcal{N}$}; \\
  \node(N2) {$\mathbb{N}$};                   \\};
\draw[->] (N1) -- (N) node[anchor=south]  {$p_{\mathbb{N}^{*}} $};
\draw[->] (N2) -- (N) node[anchor=north]  {\hspace{5mm}  $p_{\mathbb{N}}$};
\draw[<-] (N1)   -- (N2) node[anchor=east] {$\widehat{\mathbf{g}}$};
\end{tikzpicture}
\end{equation}

Next, we will introduce some basic definitions and results in contact geometry that we will need later.
See \cite[Appx. 4]{Ar89} and \cite[Ch. 5]{LM87} for more details.

\begin{definition}
Given a $n$--dimensional differentiable manifold $P$, a \emph{contact element}\index{contact!element}\index{element!contact} in $P$ is a $\left(n-1\right)$--dimensional subspace $\mathcal{H}_q\subset T_qP$.
The point $q\in P$ is called the \emph{contact point}\index{contact!point}\index{point!contact} of $\mathcal{H}_q$.

We will say that a \emph{distribution of hyperplanes}\index{distribution!hyperplanes}\index{hyperplanes!distribution} $\mathcal{H}$ in a differentiable manifold $M$ is a map $\mathcal{H}$ defined in $M$ such that for every $q\in M$ we have that $\mathcal{H}\left(q\right)=\mathcal{H}_q$ is a contact element at $q$.
\end{definition}

\begin{lemma}\label{hyperplanes-forms}
Every differentiable distribution of hyperplanes $\mathcal{H}$ can be written locally as the kernel of 1--form.
\end{lemma}

\begin{proof}
See \cite[Lem. 1.1.1]{G08} for proof.
\hfill$\Box$\bigskip

\end{proof}

It is clear that if a differentiable distribution of hyperplanes $\mathcal{H}$ is defined locally by the 1--form $\alpha\in\mathfrak{X}^{*}\left(P\right)$ then, for every non-vanishing function $f\in \mathfrak{F}\left(P\right)$ the 1--form $f\alpha$ also defines $\mathcal{H}$ since $\alpha$ and $f\alpha$ have the same kernel.

Given a distribution of hyperplanes $\mathcal{H}$ we will say that it is maximally non-integrable if for any locally defined 1-form $\eta$ such that $\mathcal{H} = \ker \eta$, then $d\eta$ is non-degenerate when restricted to $\mathcal{H}$.

\begin{definition}
A \emph{contact structure}\index{contact!structure}\index{structure!contact} $\mathcal{H}$ in a $\left(2n+1\right)$--dimensional differentiable manifold $P$ is a maximally non--integrable distribution of hyperplanes. The hyperplanes $\mathcal{H}_x \subset T_xP$ are called contact elements.
If there exists a globally defined 1-form $\eta$ defining $\mathcal{H}$, i.e., $\mathcal{H} = \ker \eta$, we will say that $\mathcal{H}$ is a \emph{cooriented contact structure}\index{cooriented!contact structure}\index{contact!structure!cooriented} and we will say that $\eta$ is a contact form.
\end{definition}

An equivalent way to determine if a distribution of hyperplanes $\mathcal{H}$ determines a contact structure is provided by the following result (see also \cite{Ar89} and \cite{Ca01}).

\begin{lemma}\label{lemma-non-degen}
Let $\mathcal{H}$ be a distribution of hyperplanes in $P$ locally defined as $\mathcal{H}=\mathrm{ker}\left(\eta\right)$, then $\left.d\eta\right|_{\mathcal{H}}$ is non--degenerated if and only if $\eta\wedge \left(d\eta\right)^n \neq 0$.
\end{lemma}

\begin{proof}
See \cite[Prop. 10.3]{Ca01} for proof.
\hfill$\Box$
\end{proof}

\begin{lemma}\label{lemma-alfa=falpa}
If $\alpha$ is a contact form in $P$, then $f\alpha$ is also a contact form for every non--vanishing differentiable function $f\in \mathfrak{F}\left(P\right)$.
\end{lemma}

\begin{proof}
See \cite[Sect. V.4.1]{LM87}.
\hfill$\Box$
\end{proof}

\subsection{Constructing the contact structure of $\mathcal{N}$} \label{sec:contact-structure-1}

Consider the tautological 1--form $\theta\in \mathfrak{X}^{*}\left(T^{*}M\right)$.
The diffeomorphism $\widehat{\mathbf{g}}:TM\rightarrow T^{*}M$ allows to carry away $\theta$ to $TM$ by pull--back.
Let $\pi_{M}^{TM}:TM\rightarrow M$ and $\pi_{M}^{T^{*}M}:T^{*}M\rightarrow M$ be the canonical projections, since $\pi_{M}^{TM}=\pi_{M}^{T^{*}M}\circ \widehat{\mathbf{g}}$, then it is verified
\[
\left(d\pi_{M}^{TM}\right)_{v}\left(\xi\right) = \left(d\pi_{M}^{T^{*}M}\right)_{\widehat{\mathbf{g}}\left(v\right)}\left(\widehat{\mathbf{g}}_{*}\left(\xi\right)\right)
\]
for all $\xi\in T_v TM$.
If we define
\begin{equation}\label{thetag}
\theta_{\mathbf{g}} = \widehat{\mathbf{g}}^{*}\theta
\end{equation}
then, using the expression \ref{tautform}, if $\xi\in T_v TM$ we have
\begin{equation*}
\left(\theta_{\mathbf{g}}\right)_v \left(\xi\right) = \widehat{\mathbf{g}}\left(v\right)\left(\left(d\pi_{M}^{T^{*}M}\right)_{\widehat{\mathbf{g}}\left(v\right)}\left(\widehat{\mathbf{g}}_{*}\left(\xi\right)\right)\right) = \mathbf{g}\left(v,\left(d\pi_{M}^{TM}\right)_{v}\left(\xi\right)\right) \, .
\end{equation*}
For a given globally hyperbolic open set $V\subset M$ equipped with coordinates $\left( x^1,\dots,x^m \right)$ such that $v\in TV$ is written as $v=v^i \frac{\partial}{\partial x^i}$, then $\left( x^i,v^i \right)$ are coordinates in $TV$.
By expression \ref{tautform-coord}, we can write
\[
\theta_{\mathbf{g}} = g_{ij}v^i dx^j \, .
\]

Let us denote by $\mathcal{H}^{TV}=\mathrm{ker}\left(\theta_{\mathbf{g}}\right)$, that is a distribution of hyperplanes in $TV\subset TM$.
This implies that $\mathrm{dim}\left(\mathcal{H}^{TV}_v\right)=2m-1$ for every $v\in TV$.

As we seen in Sect. \ref{sec:seccion-estruc-dif}, we have the chain of inclusions \ref{cadena-espacios}:
\begin{equation}
\Omega \hookrightarrow \mathbb{N}^{+}\left(C\right) \hookrightarrow \mathbb{N}^{+}\left(V\right) \hookrightarrow TV
\end{equation}
where $\Omega = \Omega^T(C) = \{ v \in \mathbb{N}^+ \mid g(v,T) = -1 \}$ for a non-vanishing timelike vector field $T$.  Observer taht if $v \in \Omega$ is the representative of the class of equivalence $[v] \in \mathbb{PN}(C)$, then clearly the following maps
\begin{equation}\label{OmegaN}
\begin{array}{ccccc}   
\Omega & \longrightarrow & \mathbb{PN}(C) & \longrightarrow & \mathcal{U} \subset \mathcal{N} \\
v & \mapsto & [v] & \mapsto & \gamma_v 
\end{array}
\end{equation}
are diffeomorphisms.

Then, we will see that the pullback of $\theta_{\mathbf{g}}$ by the inclusion $\Omega \hookrightarrow TV$ defines a 1--form $\left.\theta_{\mathbf{g}}\right|_{\Omega^{X}\left(C\right)}$, and therefore a distribution of hyperplanes, in $\Omega$.
This 1--form and its kernel can be extended from $\mathcal{U} \subset \mathcal{N}$ obtaining the 1--form $\theta_0$ looked for.

To obtain a suitable formula of  $\theta_0$ we will proceed projecting the distribution of hyperplanes in $TM$ up to $\Omega^{X}\left(C\right)$ step by step.

First, observe that the restriction of $\mathcal{H}^{TV}$ to $T\mathbb{N}^{+}\left(V\right)$, denoted by $\mathcal{H}^{\mathbb{N}^{+}\left(V\right)}$, is again a distribution of hyperplanes.
Indeed, if $c:\left(-\epsilon,\epsilon\right)\rightarrow \mathbb{N}^{+}\left(V\right)$ is a differentiable curve such that
\[
\left\{
\begin{tabular}{l}
$\alpha\left(s\right)=\pi_M^{\mathbb{N}^{+}}\left(c\left(s\right)\right)$ is a timelike curve \\
$v=c\left(0\right)\in \mathbb{N}^{+}\left(V\right)$ \\
$\xi=c'\left(0\right)\in T_v \mathbb{N}^{+}\left(V\right)$
\end{tabular}
\right.
\]
then
\[
\theta_{\mathbf{g}}\left(\xi\right)=\mathbf{g}\left(v,\alpha'\left(0\right)\right)\neq 0
\]
since $v$ is null and $\alpha'\left(0\right)$ timelike.
This implies that $\xi\notin \mathcal{H}^{TV}_v$.
So, we have that $T_v TV = span\left\{\xi\right\} \oplus \mathcal{H}^{TV}_v$ and since $span\left\{\xi\right\}\subset T_v\mathbb{N}^{+}\left(V\right)$ and $\mathcal{H}^{\mathbb{N}^{+}\left(V\right)}_v=\mathcal{H}^{TV}_v \cap T_v\mathbb{N}^{+}\left(V\right)$ then we have that
\[
\mathrm{dim}\left(\mathcal{H}^{\mathbb{N}^{+}\left(V\right)}_v\right)=2m-2
\]
therefore $\mathcal{H}^{\mathbb{N}^{+}\left(V\right)}$ is a distribution of hyperplanes in $\mathbb{N}^{+}\left(V\right)$.

The next step is to restrict $\mathcal{H}^{\mathbb{N}^{+}\left(V\right)}$ to $T\mathbb{N}^{+}\left(C\right)$, where $C$ is a Cauchy surface of $V$.
Again, as done above, if $\gamma:I\rightarrow M$ is a null geodesic verifying $\gamma\left(0\right)\in C$ and $\gamma'\left(0\right)=v\in \mathbb{N}^{+}\left(C\right)$, since the vector subspace $Z=\left\{u\in T_{v}M:\mathbf{g}\left(v,u\right)=0\right\}$ is $m-1$--dimensional and $v=\gamma'\left(0\right)\in Z$, then $\mathrm{dim}\left(Z\cap T_{\gamma\left(0\right)}C \right)=m-2$.
Hence, we can pick up a vector $\eta\in T_{\gamma\left(0\right)}C$ such that $T_{\gamma\left(0\right)}C=span\left\{\eta\right\} \oplus \left(Z\cap T_{\gamma\left(0\right)}C\right)$.
Now, we can choose a differentiable curve $c:\left(-\epsilon,\epsilon\right)\rightarrow \mathbb{N}^{+}\left(C\right)$ verifying
\[
\left\{
\begin{tabular}{l}
$c\left(0\right)=v\in \mathbb{N}^{+}\left(C\right)$ \\
$c'\left(0\right)=\kappa\in T_{v} \mathbb{N}^{+}\left(C\right)$ \\
$\left(d\pi_M^{\mathbb{N}^{+}}\right)_{v}\left(\kappa\right)=\lambda \eta$ for $\lambda\neq 0$
\end{tabular}
\right.
\]
then
\[
\theta_{\mathbf{g}}\left(\kappa\right)=\mathbf{g}\left(v,\left(d\pi_M^{\mathbb{N}^{+}}\right)_{v}\left(\kappa\right)\right)=\mathbf{g}\left(v,\lambda \eta\right)\neq 0
\]
because $\eta\notin Z$, and this shows that $\kappa\notin \mathcal{H}^{\mathbb{N}^{+}\left(V\right)}_v$.
Then $T_{v} \mathbb{N}^{+}\left(V\right) = span\left\{\kappa\right\} \oplus \mathcal{H}^{\mathbb{N}^{+}\left(V\right)}_{v}$ and since $span\left\{\kappa\right\}\subset T_v\mathbb{N}^{+}\left(C\right)$ and $\mathcal{H}^{\mathbb{N}^{+}\left(C\right)}_v=\mathcal{H}^{\mathbb{N}^{+}\left(V\right)}_v \cap T_v\mathbb{N}^{+}\left(C\right)$, then it follows
\[
\mathrm{dim}\left(\mathcal{H}^{\mathbb{N}^{+}\left(C\right)}_{v}\right)=\mathrm{dim}\left(T_v\mathbb{N}^{+}\left(C\right)\right)-1=2m-3
\]
thus $\mathcal{H}^{\mathbb{N}^{+}\left(C\right)}$ is a distribution of hyperplanes in $\mathbb{N}^{+}\left(C\right)$.

It is possible to repeat the previous argument to show that the restriction of $\mathcal{H}^{\mathbb{N}^{+}\left(C\right)}$ to $T\Omega$ defines a distribution of hyperplanes.
In fact, consider some $\eta\in T_{\gamma\left(0\right)}C$ in the same condition as before and take a differentiable curve $c:\left(-\epsilon,\epsilon\right)\rightarrow \Omega$ verifying
\[
\left\{
\begin{tabular}{l}
$c\left(0\right)=v\in \Omega$ \\
$c'\left(0\right)=\kappa\in T_{v} \Omega$ \\
$\left(d\pi_M^{\mathbb{N}^{+}}\right)_{v}\left(\kappa\right)=\lambda \eta$ for $\lambda\neq 0$
\end{tabular}
\right.
\]
then again
\[
\theta_{\mathbf{g}}\left(\kappa\right)=\mathbf{g}\left(v,\lambda \eta\right)\neq 0
\]
showing that $\kappa\notin \mathcal{H}^{\mathbb{N}^{+}\left(C\right)}_v$.
Then $T_{v} \mathbb{N}^{+}\left(C\right) = \mathrm{span}\left\{\kappa\right\} \oplus \mathcal{H}^{\mathbb{N}^{+}\left(C\right)}_{v}$ and since $\mathrm{span}\left\{\kappa\right\}\subset T_v \Omega$ then we have that
\[
\mathrm{dim}\left(\mathcal{H}^{\Omega}_{v}\right)=\mathrm{dim}\left(T_v\Omega\right)-1=2m-4
\]
thus $\mathcal{H}^{\Omega}$ is a distribution of hyperplanes in $\Omega\subset\mathbb{N}^{+}\left(C\right)$.

By this process of restriction from $TV$ to $\Omega$ we have passed $\mathcal{H}^{TV}\subset TTV$ as a distribution of hyperplanes $\mathcal{H}^{\Omega}\subset T\Omega \subset TTV$.
Moreover since $\mathcal{H}^{TV}=\mathrm{ker}\left(\theta_{\mathbf{g}}\right)$ and $\mathcal{H}^{\Omega} = T\Omega \cap \mathcal{H}^{TV}$ then
\[
\mathcal{H}^{\Omega}=\mathrm{ker}\left(\left.\theta_{\mathbf{g}}\right|_{\Omega}\right)
\]
where $\left.\theta_{\mathbf{g}}\right|_{\Omega}$ denotes the restriction of $\theta_{\mathbf{g}}$ to $\Omega$.
This fact is important in order to show that $\mathcal{H}^{\Omega}$ is a contact structure.

Then, using the diffeomorphisms in (\ref{OmegaN}), $\mathcal{H}^{\Omega}$ passes to $\mathcal{U}\subset \mathcal{N}$ as a distribution of hyperplanes of dimension $2m-4$.
Let us denote by $\mathcal{H}\subset T\mathcal{N}$ said distribution.

\begin{proposition}
If $\mathcal{U}\subset \mathcal{N}$ and $T\in \mathfrak{X}\left(M\right)$ is a given global non-vanishing timelike vector field as above, then the distribution of hyperplanes  
\begin{equation}\label{taut-form-N}
\mathcal{H}\left(\mathcal{U}\right) = \left\{\left[J\right]\in T_{\gamma}\mathcal{U}: \mathbf{g}\left(\gamma'\left(0\right),J\left(0\right) \right)= 0 \text{ with } \mathbf{g}\left(\gamma'\left(0\right), T \right)=-1 \right\}
\end{equation}
is a contact structure.
\end{proposition}

\begin{proof}
Since $\omega = -d\theta$, then taking the exterior derivative on $\theta_{\mathbf{g}}$ we obtain
\begin{equation}\label{omegag}
\omega_{\mathbf{g}} = -d\theta_{\mathbf{g}} \, ,
\end{equation}
therefore we have
\begin{equation}\label{omega_g}
\omega_{\mathbf{g}} = g_{ij}dx^j \wedge dv^i + \frac{\partial g_{ij}}{\partial x^k}v^i dx^j \wedge dx^k
\end{equation}
that clearly shows that $\omega_{\mathbf{g}}$ is a symplectic 2--form in $TM$ (notice that $\omega_{\mathbf{g}}^{ n} = \det (g_{ij}) \,  dx^1 \wedge \cdots \wedge dx^n \wedge dv^1 \wedge \cdots \wedge dv^n \neq 0$).

Consider two curves $u_n\left(s\right)=u_n^i\left(s\right)\left(\frac{\partial}{\partial x^i}\right)_{\alpha_n\left(s\right)}\in TM$ where $n=1,2$ such that
\begin{equation*}
\begin{tabular}{l}
$\alpha'_n\left(s\right)= a_n^i\left(s\right)\left(\frac{\partial}{\partial x^i}\right)_{\alpha_n\left(s\right)}$ \\
$u'_n\left(s\right)=a_n^i\left(s\right)\left(\frac{\partial}{\partial x^i}\right)_{u_n\left(s\right)}+ \frac{d u^i_n}{ds}\left(s\right)\left(\frac{\partial}{\partial v^i}\right)_{u_n\left(s\right)}$
\end{tabular}
\end{equation*}
and recall that
\[
\frac{Du_n}{ds} = \left(\frac{d u^k_n}{ds}+ \Gamma^k_{ij}a^i_n u^j_n \right)\left(\frac{\partial}{\partial x^k}\right)_{\alpha_n}
\]
calling $\frac{D^k u_n}{ds}= \frac{d u^k_n}{ds}+ \Gamma^k_{ij}a^i_n u^j_n$ the $k$--th component of $\frac{Du_n}{ds}$.
If $u=u_1\left(0\right)=u_2\left(0\right)$ and $\xi_n=u'_n\left(0\right)$ for $n=1,2$, then we have that:
\begin{align}
\omega_{\mathbf{g}}\left(\xi_1,\xi_2\right) &= g_{ij}a^i_1 \frac{D^j u_2}{ds}- g_{ij}a^j_2 \frac{D^i u_1}{ds} + \left(g_{kl}\Gamma^l_{ji} - g_{jl}\Gamma^l_{ki} +\frac{\partial g_{ij}}{\partial x^k}- \frac{\partial g_{ik}}{\partial x^j}\right)u^i a^j_1 a_2^k  = \nonumber \\
&= g_{ij}a^i_1 \frac{D^j u_2}{ds}- g_{ij}a^j_2 \frac{D^i u_1}{ds} = \mathbf{g}\left(\alpha'_1\left(0\right), \frac{D u_2}{ds}\left(0\right) \right) - \mathbf{g}\left(\alpha'_2\left(0\right), \frac{D u_1}{ds}\left(0\right) \right) \, . \nonumber
\end{align}

Since the exterior derivative commutes with the restriction to submanifolds, then
\[
\left.\omega_{\mathbf{g}}\right|_{\Omega}= -\left.\left(d\theta_{\mathbf{g}}\right)\right|_{\Omega}=-\left.d\left(\theta_{\mathbf{g}}\right|_{\Omega}\right)
\]
Proposition \ref{prop00254} permit to transmit $\left.\theta_{\mathbf{g}}\right|_{\Omega},\left.\omega_{\mathbf{g}}\right|_{\Omega}$ to $\mathcal{L}\left(\gamma_u\right)$ pointwise.
Calling $\theta_0$ and $\omega_0$ the resulting forms, then for $\left[J\right],\left[J_1\right],\left[J_2\right]\in \mathcal{L}\left(\gamma_u\right)$ we have
\begin{equation*}
\theta_{0}\left(\left[J\right]\right) = \mathbf{g}\left(\gamma'_u\left(0\right),J\left(0\right) \right)
\end{equation*}
where $\gamma_u$ is parametrized such that $\gamma'_u\left(0\right)\in \Omega$, and
\begin{equation}\label{symp-form-N}
\omega_{0}\left(\left[J_1\right],\left[J_2\right]\right) = \mathbf{g}\left(J_1\left(0\right), J'_2\left(0\right) \right) - \mathbf{g}\left(J_2\left(0\right), J'_1\left(0\right) \right)
\end{equation}

In order to prove that $\mathcal{H}$ is a contact structure, we will show that $\left.\omega_0\right|_{\mathcal{H}}$ is non--degenerated.
Consider $\left[J_1\right],\left[J_2\right]\in \mathcal{H}$, then the initial values of $J_1$ and $J_2$ in expression \ref{symp-form-N} verify
\begin{equation}\label{Jac-Init-Values}
\mathbf{g}\left(J_i\left(0\right),\gamma'_u\left(0\right)\right)=0 \,\, ; \qquad 
\mathbf{g}\left(J'_i\left(0\right),\gamma'_u\left(0\right)\right)=0 \, .
\end{equation}
for $i=1,2$, that is $J_i\left(0\right),J'_i\left(0\right)\in \left\{\gamma'_u\right\}^{\perp}=\left\{v\in T_{\gamma_u\left(0\right)}M:\mathbf{g}\left(v,\gamma'_u\left(0\right)\right)=0\right\}$.

For a given $\left[J_1\right]\in \mathcal{H}$, if $\omega_{0}\left(\left[J_1\right],\left[J_2\right]\right) = 0$ for all $\left[J_2\right]\in \mathcal{L}\left(\gamma_u\right)$, then in particular, also for $\left[J_2\right]$ verifying $J'_2\left(0\right)=0$, we have
\[
\omega_{0}\left(\left[J_1\right],\left[J_2\right]\right) = 0 \Rightarrow  \mathbf{g}\left(J_2\left(0\right), J'_1\left(0\right) \right)=0
\]
Since $J'_1\left(0\right)\in \left\{\gamma'_u\right\}^{\perp}$, the only vector $J'_1\left(0\right)$ such that $\mathbf{g}\left(J_2\left(0\right), J'_1\left(0\right) \right)=0$ for all $J_2\left(0\right)\in \left\{\gamma'_u\right\}^{\perp}$ is, by definition of $\left\{\gamma'_u\right\}^{\perp}$, the vector $J'_1\left(0\right)=0\left(\mathrm{mod}\gamma'_u\right)$.

On the other hand, for $\left[J_2\right]$ verifying $J_2\left(0\right)=0$ we have
\[
\omega_{0}\left(\left[J_1\right],\left[J_2\right]\right) = 0 \Rightarrow  \mathbf{g}\left(J_1\left(0\right), J'_2\left(0\right) \right)=0
\]
and again, since $J_1\left(0\right)\in \left\{\gamma'_u\right\}^{\perp}$ then the only vector $J_1\left(0\right)$ such that $\mathbf{g}\left(J_1\left(0\right), J'_2\left(0\right) \right)=0$ for all $J'_2\left(0\right)\in \left\{\gamma'_u\right\}^{\perp}$ is $J_1\left(0\right)=0\left(\mathrm{mod}\gamma'_u\right)$.

Thus, the only $\left[J_1\right]\in \mathcal{H}$ such that $\omega_{0}\left(\left[J_1\right],\left[J_2\right]\right) = 0$ for all $\left[J_2\right]\in \mathcal{H}$ is $J_1=0\left(\mathrm{mod}\gamma'_u\right)$, therefore $\left.\omega_{0}\right|_{\mathcal{H}}$ is non--degenerated.
This shows that $\mathcal{H}$ is a contact structure in $\mathcal{N}$.
\hfill$\Box$\bigskip

\end{proof}

Let us take $\gamma\in\mathcal{U}\cap\mathcal{V}$, since in general $\frac{d}{dt}\mathbf{g}\left(\gamma'\left(t\right),T \left(\gamma\left(t\right)\right)\right)\neq 0$, then there are different parameter for $\gamma$ in order to write $\mathcal{H}\left(\mathcal{U}\right)$ and $\mathcal{H}\left(\mathcal{V}\right)$ as in expression \ref{taut-form-N}.
If we consider that $\gamma=\gamma\left(t\right)$ and $\overline{\gamma}=\overline{\gamma}\left(\tau\right)$ are the parametrizations of $\gamma\in \mathcal{U}\cap\mathcal{V}$ such that $\overline{\gamma}\left(\tau\right)=\gamma\left(a\tau +b\right)$ verifying
\[
\mathbf{g}\left(\gamma'\left(0\right),T\right)= -1 \, \, ; \qquad  
\mathbf{g}\left(\overline{\gamma}'\left(0\right),T\right)= -1 \, .
\]
By definition of $\mathcal{J}_L\left(\overline{\gamma}\right)$, we have that $\mathbf{g}\left(\overline{J}\left(\tau\right),\overline{\gamma}'\left(\tau\right)\right)$ is constant, therefore
\[
\mathbf{g}\left(\overline{J}\left(0\right),\overline{\gamma}'\left(0\right)\right)=\mathbf{g}\left(\overline{J}\left(-b/a\right),\overline{\gamma}'\left(-b/a\right)\right) = a \mathbf{g}\left(J\left(0\right),\gamma'\left(0\right)\right)
\]
as we have seen in remark \ref{remark-Jacobi-init-val}, whence since $\overline{\gamma}\left(-b/a\right)=\gamma\left(0\right)$ we have
\[
\mathbf{g}\left(\overline{J}\left(-b/a\right),\overline{\gamma}'\left(-b/a\right)\right) = 0 \Leftrightarrow  \mathbf{g}\left(J\left(0\right),\gamma'\left(0\right)\right) = 0
\]
The same argument above is valid to prove that $\mathcal{H}_{\gamma}$ does not depends on the timelike vector field used to define $\Omega$, because it only affects to the parametrization of $\gamma$.
This shows that $\mathcal{H}_{\gamma}$ is well defined and does not depends on the neighbourhood used in its construction.    In addition, using Lemma \ref{JJbar}, can be shown that this contact structure does not depend on the auxiliary metric $\mathbf{g}$ selected in the conformal class $\mathcal{C}$.

At this point, we may consider a covering $\left\{ \mathcal{U}_{\delta} \right\}_{\delta\in I}\subset \mathcal{N}$ and, for any $\delta\in I$, consider the local 1--form $\theta_{0}^{\delta}$ defining the contact structure $\mathcal{H}$ as before.
If we take a partition of unity $\left\{ \chi_{\delta} \right\}_{\delta\in I}$ subordinated to the covering $\left\{ \mathcal{U}_{\delta} \right\}_{\delta\in I}$ then we can define a global 1--form by:
\[
\theta_0 \left(\left[J\right]\right)= \sum_{\delta\in I} \chi_{\delta}\left(\left[J\right]\right)\cdot\theta_{0}^{\delta}\left(\left[J\right]\right)
\]
then the contact structure $\mathcal{H}$ is cooriented since $\theta_0$ is globally defined and, by Lemma \ref{lemma-non-degen}, remains maximally non-integrable.

In the following section we will provide a slightly more intrinsic construction of the canonical contact structure on the space of light rays based on symplectic reduction techniques.


\section{The contact structure in $\mathcal{N}$ and symplectic reduction}

Finally, in this section, we will illustrate the construction of the contact structure in $\mathcal{N}$ in a equivalent but more elegant way as done in section~\ref{sec:contact-structure-1}.




\subsection{The coisotropic reduction of $\mathbb{N}^+$ and the symplectic structure on the space of scaled null geodesics $\mathcal{N}_s$}

The celebrated Theorem of Marsden--Weinstein \cite{MW74} claims that a $2m$--dimensional symplectic manifold $P$, in which a Lie group $G$ acts preserving the symplectic form $\omega$ and possessing an equivariant momentum map, can be reduced into another $2(m-r)$--dimensional symplectic manifold $P_{\mu}$, called the Marsden-Weinstein reduction of $P$ with respect to $\mu$, under the appropriate conditions where $\mu$ is an element in the dual of the Lie algebra of the group $G$ and $r$ is the dimension of the coadjoint orbit passing through $\mu$.

The purpose of this section is to show that it is possible to derive the canonical contact structure on the space of light rays by a judiciously use of Marsden-Weinstein reduction when the geodesic flow defines an action of the Abelian group $\mathbb{R}$ in $TM$.   However we will choose a different, simpler, however more general path here.  Simpler in the sense that we will not need the full extent of MW reduction theorem, but a simplified version of it obtained when restricted to scalar momentum maps, but more general in the sense that it will not be necessary to assume the existence of a group action.   Actually the setting we will be using is a particular instance of the scheme called generalized symplectic reduction (see for instance \cite[Ch.7.3]{Ca14} and references therein).

The result we are going to obtain is based on the following elementary algebraic fact.  Let $(E,\omega)$ be a linear symplectic space.  Let $W \subset E$ be a linear subspace.  We denote by $W^\perp$ the symplectic orthogonal to $W$, i.e., $W^\perp = \{ u \in E \mid \omega (u,w) = 0 \, , \forall w \in W  \}$.   A subspace $W$ is called coisotropic if $W^\perp \subset W$.   It is easy to show that for any subspace $W$, $\dim W + \dim W^\perp = \dim E$.    Hence it is obvious that if $H$ is a linear hyperplane, that is a linear subspace of codimension 1, then $H$ is coistropic (clearly because $\omega_H$ is degenerate, then $H \cap H^\perp \neq \{ \mathbf{0} \} $ and because $H^\perp$ is one-dimensional, then $H^\perp \subset H$).  Moreover the quotient space $H/H^\perp$ inherits a canonical symplectic form $\bar{\omega}$ defined by the expression:
$$
\bar{\omega}(u_1 + H^\perp, u_2 + H^\perp ) = \omega (u_1, u_2) \, , \qquad \forall u_1.u_2 \in H 	\, .
$$

The linear result above has a natural geometrical extension:

\begin{theorem}\label{coisotropic}  Let $(P,\omega)$ be symplectic manifold and $i\colon S \to P$ be a hypersurface, i.e., a codimension 1 immersed manifold.  Then:

\begin{enumerate}
\item[i.]  The symplectic form $\omega$ induces a 1-dimensional distribution $K$ on $S$, called the characteristic distribution of $\omega$, defined as $K_x = \ker i^*\omega_x = T_xS^\perp \subset T_xS$.

\item[ii.]  If we denote by $\mathcal{K}$ the 1-dimensional foliation defined by the distribution $K$ and $\overline{S} = S/\mathcal{K}$ has the structure of a quotient manifold, i.e., the canonical projection map $\rho \colon S \to S/\mathcal{K}$ is a submersion, then there exists a unique symplectic form $\bar{\omega}$ on $\overline{S}$ such that $\rho^*\bar{\omega} = i^* \omega$.

\item[iii.]  If $\omega = -d\theta$ and there exists $\bar{\theta}$ a 1-form on $\overline{S}$ such that $\rho^*\bar{\theta} = i^*\theta$, then $\bar{\omega} = -d \bar{\theta}$.

\end{enumerate}
\end{theorem}

\begin{proof}
The proof of (i) is just the restriction of the algebraic statements above to $W = T_xS  \subset E = T_xP$.

To proof (ii), notice that a vector tangent to the leaves of $\mathcal{K}$ is in the kernel of $i*\omega$, then for any vector field $X$ on $S$ tangent to the leaves of $\mathcal{K}$, i.e, projectable to $0$ under $\rho$, we have $i_X(i^*\omega) = 0$, and $\mathcal{L}_X (i^*\omega) = 0$, then the 2-form $i^*\omega$ is projectable under $\rho$.

The statement (iii) is trivial because $\rho^*\bar{\omega} = i^*\omega = i^*(-d\theta) = - di^*\theta = -d\rho^*\bar{\theta} = \rho^*(-d\bar{\theta})$ and $\rho$ is a submersion.  
\hfill$\Box$\bigskip

\end{proof}

The previous theorem states that any hypersurface on a symplectic manifold is coisotropic and that, provided that the quotient space is a manifold, the space of leaves of its characteristic foliation, inherits a symplectic structure. Such space of leaves is thus the reduced symplectic manifold we are seeking for and it will be called the coisotropic reduction of the hypersurface $S$.    In addition to the previous reduction mechanism, we will also use the following passing to the quotient mechanism for hyperplane distributions.

\begin{theorem}\label{contact_red}  Let $(P, \omega = d\theta )$ be an exact symplectic manifold and $\pi \colon P \to N$ be a submersion on a manifold of dimension $\dim P - 1$ and such that it projects the hyperplane distribution $H = \ker \theta$, that is there exists a hyperplane distribution $H^N$ in $N$ such that for any $x\in P$, $\pi_*(x)H_x = H_{\pi(x)}^N$.  Then $H^N$ defines a contact structure on $N$.
\end{theorem}

\begin{proof}  Notice that necessarily, $\ker \pi_*(x) = H_x^\perp$ and $\omega$ induces a symplectic form $\bar{\omega}_x$ in $H/H^\perp$ because Thm. \ref{coisotropic}.  Moreover $H_x/H_x^\perp \cong H_{\pi(x)}^N$ and it inherits a symplectic form $\bar{\omega}_x$.    Finally, if we pick up a local section $\sigma$ of the submersion $\pi$; then the 1-form $\sigma^*\theta$ is such that $H^N = \ker \sigma^*\theta$ and $d(\sigma^*\theta)$ coincides with the symplectic form $\bar{\omega}_x$ when restricted to $H_x^N$.
\hfill$\Box$\bigskip

\end{proof}

The two previous results, Thm. \ref{coisotropic} and \ref{contact_red}, hold the key to understand how the quotient space $\mathcal{N}$ inherits a canonical contact structure.   Consider again a spacetime $(M,\mathbf{g})$ and the canonical identification provided by the metric $\hat{\mathbf{g}} \colon \widehat{T}M \to \widehat{T}^*M$ (which is just the Legendre transform corresponding to the Lagrangian function $L_{\mathbf{g}}(x, v) = \frac{1}{2} \mathbf{g}_x(v,v)$ on $TM$).    As we discussed at the beginning of Sect. \ref{sec:contact}, Eqs. (\ref{thetag}), (\ref{omegag}), we can pull-back the canonical 1-form $\theta$ on $T^*M$ along $\hat{\mathbf{g}}$ as well the canonical symplectic structure $\omega$ (Sect. \ref{sec:T*M}), that is, we obtain:
$$
\theta_g = \widehat{\mathbf{g}}^*\theta \, , \qquad \omega_g = \widehat{\mathbf{g}}^* \omega = - d\theta_g \, ,
$$
and $(\widehat{T}M, \omega_{\mathbf{g}})$ becomes a symplectic manifold.
Moreover $\mathbb{N}^+ \subset \widehat{T}M$ defines an hypersurface, hence by Thm. \ref{coisotropic} we can construct its coisotropic reduction.

We will denote by $\mathcal{N}_s$ the space of equivalence classes of future-oriented null geodesics that differ by a translation of the parameter.  Thus two parametrized null geodesics $\gamma_1 (t)$, $\gamma_2(t')$ are equivalent if there exists a real number $s$ such that $\gamma_2(t' ) = \gamma_1(t+s)$.   The equivalence class of null geodesics containing the parametrized geodesic $\gamma(t)$ such that $\gamma'(0) = v$ will be denoted by $\gamma_v$.

Clearly there is a natural projection $\pi\colon \mathcal{N}_s \to \mathcal{N}$ 
defined as $\pi (\gamma_v) = [\gamma]$.
The space $\mathcal{N}_s$ is sometimes called the space of \emph{scaled null geodesic}\index{space!scaled null geodesic}\index{scaled!null geodesic}\index{null!scaled -- geodesic} and describes equivalence classes of null geodesics distinguishing different scale parametrizations.

\begin{theorem}  Let $(M, \mathbf{g})$ be a spacetime, then:
\begin{enumerate}
\item[i.]  The characteristic distribution $K  = \ker \omega_{\mathbf{g}}\mid_{\mathbb{N}^+}$ is generated by the restriction of the geodesic spray $X_{\mathbf{g}}$ to $\mathbb{N}^+$ and $\mathbb{N}^+/K$ can be identified naturally with the space of scaled null geodesics $\mathcal{N}_s$.
\item[ii.]  If $M$ is strongly causal, $\mathcal{N}_s$ is a quotient manifold of $\mathbb{N}^+$, and it becomes a symplectic manifold with the canonical reduced symplectic structure obtained by coisotropic reduction of $\omega_{\mathbf{g}}$.
\end{enumerate}

\end{theorem}

\begin{proof}   To prove [i] we just check that $ {\omega_{\mathbf{g}}}_v (X_{\mathbf{g}}, Y) = dL_v(Y) = Y_v(L) = 0$ for all  $Y\in T_v(\mathbb{N}^+)$ because $\mathbb{N}^+ = L^{-1}(\mathbf{0})$ 
where $L$ is the Lagrangian function $L\left(u\right)=\frac{1}{2}\mathbf{g}\left(u,u\right)$

Notice that the flow $\varphi_t$ of the geodesic spray $X_{\mathbf{g}}$ is such that $\varphi_s(\gamma(t)) = \gamma(t+s)$ where $\gamma(t)$ is a parametrized geodesic.  Then the quotient $\mathbb{N}^+/K$ corresponds exactly to the notion of scaled null geodesic before.  We will denote, as before, by $\rho \colon \mathbb{N}^+ \to \mathcal{N}_s$ the canonical projection and, with the notations above, we get simply that $\rho(v) = \gamma_v$.

As $M$ is strongly causal, the proof of [ii] mimics the proof of Prop. \ref{prop00150}
.  Hence because [ii] in Thm. \ref{coisotropic}, we  conclude that the quotient manifold inherits a canonical symplectic structure by coisotropic reduction of $\omega_{\mathbf{g}}$.
\hfill$\Box$\bigskip

\end{proof}


\endDocument 


\subsection{The contact structure of $\mathcal{N}$ and the reduction of the space of scaled null geodesics $\mathcal{N}_s$}\label{sec:contact_Ns}

We will prove first that the space of light rays $\mathcal{N}$ is the base manifold of a principal bundle with structural group $\mathbb{R}^+$ whose total space is the space of scaled null geodesics $\mathcal{N}_s$.
Notice that the Euler vector field $\Delta$ on $\mathbb{N}^+$ is projectable under the map $\rho \colon \mathbb{N}^+ \to \mathcal{N}_s$. The projected vector field will be denoted by $\Delta_s$ and its flow is defined by $\Phi_t(\gamma_v) = \gamma_{tv}$.   Clearly $\rho(tv) = \gamma_{tv}$, $\rho\circ \varphi_t = \Phi_t\circ \rho$, hence $\rho_*\Delta = \Delta_s$.

\begin{lemma}\label{Ns-N-prinbundle}

The map $\pi:\mathcal{N}_s \rightarrow \mathcal{N}$ is a principal bundle with structural group the multiplicative group $\mathbb{R}^+$.
\end{lemma}

\begin{proof}
We will show that the flow $\Phi:\mathbb{R}^+\times \mathcal{N}_s \rightarrow \mathcal{N}_s$ defined by the vector field $\Delta_s \in\mathfrak{X}\left(\mathcal{N}_s\right)$ defines a free and proper right action.
It is clear that $\Phi$ is a right action and it can be written by
\[
\Phi\left(t,\gamma_v\right)=\gamma_{tv}
\]
where $\gamma_u$ denotes the null geodesic defined by the null vector $u\in \mathbb{N}^+$.
It is well-known that for any non-zero $\lambda\in \mathbb{R}$ it is verified that $\gamma_{\lambda u}\left(s\right)=\gamma_u\left(\lambda s\right)$.
Then, the equality $\Phi\left(t,\gamma_v\right)=\gamma_v$ implies that $tv=v$, whence $t=1$, and the action is free.

Now, consider two sequences $\left\{\gamma_{v_n}\right\}$ and $\left\{\Phi_{t_n}\left(\gamma_{v_n}\right)\right\}$ converging to $\gamma_v$ and $\gamma_u$ respectively.
Since $\Phi_{t_n}\left(\gamma_{v_n} \right) = \gamma_{t_n v_n}$ and again $\gamma_{t_n v_n}\left(s\right)=\gamma_{v_n}\left(t_n s\right)$
then $\gamma_u$ and $\gamma_v$ have the same image as a parametrized curve in $M$, then for every $s$ we have  $\gamma_u\left(s\right)=\gamma_{v}\left(\overline{t}s\right)$ for some $\overline{t}\in \mathbb{R}^+$.
So we have that $u=\overline{t}v$ and hence $t_n v_n \mapsto \overline{t}v$.
Since $v\neq 0$ and $\left\{
\begin{tabular}{l}
$v_n \mapsto v$ \\
$t_n v_n \mapsto \overline{t}v$
\end{tabular}
\right.$ then we have that $t_n \mapsto \overline{t}$.
This shows that the action $\Phi$ is proper.
Then we have that $\pi:\mathcal{N}_s \rightarrow \mathcal{N}$, $\pi (\gamma_v) = [\gamma]$, is a principal bundle with structural group $\mathbb{R}^+$.
\hfill$\Box$\bigskip

\end{proof}

The following theorem shows that the distribution of hyperplanes defined by the canonical 1-form $\bar{\theta}$ actually projects down to $\mathcal{N}$ defining a canonical contact structure.  First notice that since $\pi^*(\mathcal{L}_{\Delta_s} \overline{\omega}) =  \mathcal{L}_\Delta \pi^*\overline{\omega} = i^*\mathcal{L}_\Delta \omega = \pi^*\overline{\omega}$, then
$$
\mathcal{L}_{\Delta_s} \overline{\omega} = \overline{\omega} \, ,
$$
and $\Delta_s$ is a Liouville vector field. Then consider the 1--form
\begin{equation}\label{bartheta}
\overline{\theta} = -i_{\Delta_s}\overline{\omega} \, ,
\end{equation} in $\mathcal{N}_s$.   Clearly $i_{\Delta_s}\bar{\theta} = 0$ and
$d\bar{\theta} = - \bar{\omega}$.

\begin{theorem}\label{theo-contact-structure}
The 1-form $\bar{\theta}$ in $\mathcal{N}_s$ induces a distribution of hyperplanes $\overline{\mathcal{H}} = \ker \bar{\theta}$.  Moreover there exists a distribution of hyperplanes $\mathcal{H}$ in $\mathcal{N}$ such that $\pi_*(\overline{H}) = \mathcal{H}$ that defines a contact structure.  Hence $\mathcal{N}$ is equipped with a canonical contact structure.
\end{theorem}

\begin{proof}   Because of Eq. (\ref{bartheta}), $\mathcal{L}_{\Delta_s} \bar{\theta} = i_{\Delta_s} d\bar{\theta} = -\bar{\theta}$, then $(\Phi_t)^*\bar{\theta} = t \bar{\theta}$.   Then at each point $[\gamma] \in \mathcal{N}$, the hyperplanes obtained by projecting the hyperplanes $\mathcal{H}_{\gamma_{tv}}$ are the projection of the kernels of the family of proportional covectors $t\bar{\theta}_v$, hence they are the same.
Then because of Thm. \ref{contact_red}, this implies that $\mathcal{H}$ is a contact structure in $\mathcal{N}$.
\hfill$\Box$\bigskip

\end{proof}

\begin{proposition}\label{prop-theta_0}
Let $\pi:\mathcal{N}_s \rightarrow \mathcal{N}$ be the principal bundle of lemma \ref{Ns-N-prinbundle}, then there exists a 1--form $\theta_0$ in $\mathcal{N}$ such that $\mathcal{H} = \ker\theta_0$ and the contact structure is coorientable.
\end{proposition}

\begin{proof}  Because the structure group of the principal bundle $\mathcal{N}_s \to \mathcal{N}$ is contractible, then is trivial. Hence there exists a smooth global section $\sigma$, then we can define the 1-form $\theta_0 = \sigma^*\bar{\theta}$ and, clearly, $\ker\theta_0 = \mathcal{H}$.
\hfill$\Box$\bigskip

\end{proof}

Notice that the bundle $\mathcal{N}_s \to \mathcal{N}$ is trivial because the group $\mathbb{R}^+$ is contractible, however if we were considering the space of non-oriented (future or past) unparametrized null geodesics instead, such space would be obtained by quotienting $\mathcal{N}_s$ with respect to the group $\mathbb{R}_* = \mathbb{R}-\{Ê0 \}$ which is not contractible and the corresponding contact structure will not be coorientable.

It is also noticeable, that like in the construction at the end of Sect. \ref{sec:contact}, the 1-form $\theta_0$ describing the canonical contact structure on $\mathcal{N}$ is not canonically determined (even if the hyperplane distribution is), in former case the 1-form $\theta_0$ depends on the choice of a partition of the unity, in the present case, it depends on a section of a principal bundle.

\bigskip

Next, we will look for an expression for the local contact forms defining the contact structure $\mathcal{H}\subset T\mathcal{N}$.
Recall that a coordinate chart $\psi:\mathcal{U}\subset \mathcal{N}\rightarrow \mathbb{R}^{2m-3}$ can be defined via the diffeomorphism $\mathcal{U}\rightarrow \Omega^{T}\left(C\right)$ of diagram \ref{cadena-espacios}, where $\Omega^{T}\left(C\right)$ is an embedded submanifold of $TV\subset TM$ with $V\subset M$ a globally hyperbolic and causally convex open set with Cauchy surface $C$.
Then we have the following diagram
\begin{equation}\label{diagrama008}
\begin{tikzpicture}[every node/.style={midway}]
\matrix[column sep={9em,between origins},
        row sep={3em}] at (0,0)
{  \node(N) {$\mathcal{N}\supset \mathcal{U}$}; & \node(N1) {$TV$}; \\
   \node(B) {$\mathbb{R}^{2m-3}\supset B_0$}; & \node(B1) {$B\subset \mathbb{R}^{2m}$};   \\};
\draw[->] (N) -- (N1) node[anchor=south]  {$\overline{z} $};
\draw[->] (B) -- (B1) node[anchor=north]  {$z$};
\draw[->] (N)   -- (B) node[anchor=east] {$\psi$};
\draw[->] (N1)   -- (B1) node[anchor=west] {$\phi$};
\end{tikzpicture}
\end{equation}
where $\overline{z}=\phi^{-1}\circ z \circ \psi $.
The image of the embedding $\overline{z} $ is contained in $\mathbb{N}^{+}\left(C\right)$, and moreover if $p_{\mathbb{N}}:\mathbb{N}\rightarrow \mathcal{N}$ is the canonical projection, then $p_{\mathbb{N}}\circ \overline{z}\left(\left[\gamma\right]  \right)=\left[\gamma\right]$ for all $\left[\gamma\right]\in \mathcal{U}\subset\mathcal{N}$.
Then $\overline{z}$ is a local section of $p_{\mathbb{N}}$.

By proposition \ref{prop-theta_0} and theorem \ref{theo-contact-structure}, we have that for $\xi \in T\mathbb{N}$
\[
\theta_{\alpha}\left(\xi\right) = \left( \theta_0 \right)_{p_{\mathbb{N}^{*}}\left(\alpha\right)}\left( \left( dp_{\mathbb{N}^{*}} \right)_{\alpha}\left(\xi\right) \right)
\]
then, by diagram in \ref{diagrama006}, $p_{\mathbb{N}}=p_{\mathbb{N}^{*}}\circ \widehat{\mathbf{g}}$, and hence we can write for $J\in T_{\left[\gamma\right]}\mathcal{N}\subset T\mathcal{U}$
\[
\theta_{\widehat{\mathbf{g}}\circ \overline{z}\left(\left[\gamma\right] \right)}\left( d\left( \widehat{\mathbf{g}}\circ \overline{z}\right)_{\left[\gamma\right]}\left(J \right)  \right) = \left( \theta_0 \right)_{\left[\gamma\right]}\left(J\right)
\]
On the other hand, by definition of the tautological 1--form $\theta$ and since $\pi^{TM}_{M}=\pi^{T^{*}M}_{M}\circ \widehat{\mathbf{g}}$, we have
\begin{align*}
\theta_{\widehat{\mathbf{g}}\circ \overline{z}\left(\left[\gamma\right] \right)}\left(  d\left( \widehat{\mathbf{g}}\circ \overline{z}\right)_{\left[\gamma\right]}\left(J \right)  \right) &= \widehat{\mathbf{g}}\circ \overline{z}\left(\left[\gamma\right] \right) \left(  d\left( \widehat{\mathbf{g}}\circ \overline{z}\right)_{\left[\gamma\right]}\left(J \right)  \right) = \\
&=\mathbf{g}\left( \left(d\pi^{T^{*}M}_{M}\right)_{\widehat{\mathbf{g}}\circ\overline{z}\left(  \left[\gamma\right] \right)}\left( d\left( \widehat{\mathbf{g}}\circ \overline{z}\right)_{\left[\gamma\right]}\left(J \right) \right), \overline{z}\left(\left[\gamma\right] \right)\right) =  \\
&= \mathbf{g}\left( \left(d\pi^{TM}_{M}\right)_{\overline{z}\left(  \left[\gamma\right] \right)}\left( d\overline{z}_{\left[\gamma\right]}\left(J \right) \right), \overline{z}\left(\left[\gamma\right] \right)\right)= \\
&=\mathbf{g}\left( J\left(0 \right), \gamma'\left(0\right)\right)
\end{align*}
where $\overline{z}\left(\left[\gamma\right] \right)\in \mathbb{N}^{+}\left(C\right)$ is a vector defining the light ray $\left[\gamma\right]\in \mathcal{U}$, so we have considered the null geodesic $\gamma$ such that $\gamma'\left(0\right)=\overline{z}\left(\left[\gamma\right] \right)$.
On the other hand, observe that if $J=\Gamma'\left(0\right)\in T_{\left[\gamma\right]}\mathcal{N}$ where $\Gamma$ is a smooth curve in $\mathcal{N}$ with $\Gamma\left(0\right)=\left[\gamma\right]$, then
\[
\left(d\pi^{TM}_{M}\right)_{\overline{z}\left(  \left[\gamma\right] \right)}\left( d\overline{z}_{\left[\gamma\right]}\left(J \right) \right)=
\left(d\pi^{TM}_{M}\right)_{\overline{z}\left(  \left[\gamma\right] \right)}\left( d\overline{z}_{\left[\gamma\right]}\left(\Gamma'\left(0\right) \right) \right)= \left(\pi^{TM}_{M}\circ \overline{z} \circ \Gamma \right)'\left(0 \right)
\]
where $\pi^{TM}_{M}\circ \overline{z} \circ \Gamma$ is the curve in $M$ where $\overline{z} \circ \Gamma\subset \mathbb{N}\left(C\right)$ rest.
By lemma \ref{lemmaDC92}, we have that $\left(\pi^{TM}_{M}\circ \overline{z} \circ \Gamma \right)'\left(0 \right)=J\left(0\right)$ and therefore we claim that
\[
\left( \theta_0 \right)_{\left[\gamma\right]}\left(J\right) = \mathbf{g}\left( J\left(0 \right), \gamma'\left(0\right)\right)
\]
and then
\[
J\in \mathcal{H} \Longleftrightarrow \mathbf{g}\left( J\left(0 \right), \gamma'\left(0\right)\right)=0.
\]
It is clear that this characterization does not depends neither on the representative metric of the conformal class $\mathcal{C}$ nor on the parametrization of $\gamma$ in virtue of lemma \ref{lem00215}.

Again, since the expression of the local 1--form $\theta_{0}$ defining the contact structure $\mathcal{H}$ coincides with the one constructed in section~\ref{sec:contact-structure-1}, then the same used argument to show that $\mathcal{H}$ is cooriented remains valid.